\documentclass{article}
\usepackage[utf8]{inputenc}
\usepackage[small,compact]{titlesec}
\usepackage{amsmath,bm,mathtools,amsfonts,mathrsfs,amssymb,amsthm,mathabx,setspace}
\usepackage{graphicx,caption,epsfig,subfigure,epsfig,wrapfig,sidecap}
\usepackage{url,color,verbatim,algorithmicx}
\usepackage[sort,compress]{cite}

\usepackage{bm,float}
\usepackage[top=1in,bottom=1in,left=1in,right=1in]{geometry}

\usepackage[ruled,boxed]{algorithm}
\usepackage[scaled]{helvet}
\usepackage[T1]{fontenc}
\usepackage[bookmarks=true, bookmarksnumbered=true, colorlinks=true,   pdfstartview=FitV,
linkcolor=blue, citecolor=blue, urlcolor=blue]{hyperref}
\usepackage{multirow, multicol,makecell, colortbl, hhline}
\usepackage{booktabs} 
\usepackage{enumitem} 
\usepackage{etoc}          

\newcommand{\innerp}[1]{\langle{#1}\rangle}
\def\Gap{\mathrm{Gap}}
\def\deltat{\delta}

\usepackage[noend]{algpseudocode}

\newtheorem{theorem}{Theorem}[section]
\newtheorem{assumption}[theorem]{Assumption}
\newtheorem{remark}[theorem]{Remark}
\newtheorem*{remark*}{Remark}

\newtheorem{lemma}[theorem]{Lemma}

\newcommand{\R}{\mathbb{R}}

\newcommand{\mcB}{\mathcal{B}}

\newcommand{\mcF}{\mathcal{F}}

\newcommand{\mcM}{\mathcal{M}}

\usepackage{xcolor}

\newcommand{\helen}[1]{\textcolor{blue}{#1}}

\def \Xb{\mathbf{X}}

\def \E{\mathbb{E}}

\def\mB{{\bf B}}
\def\mX{{\bf X}}     

\numberwithin{equation}{section}

\title{
ISALT: Inference-based schemes adaptive to large time-stepping \\ for locally Lipschitz ergodic systems 
}
\author{Xingjie Li, Fei Lu, and Felix X.-F. Ye }
\date{}

\begin{document}

\maketitle


\begin{abstract}
    Efficient simulation of SDEs is essential in many applications, particularly for ergodic systems that demand efficient simulation of both short-time dynamics and large-time statistics. However, locally Lipschitz SDEs often require special treatments such as implicit schemes with small time-steps to accurately simulate the ergodic measure. We introduce a framework to construct inference-based schemes adaptive to large time-steps (ISALT) from data, achieving a reduction in time by several orders of magnitudes. The key is the statistical learning of an approximation to the infinite-dimensional discrete-time flow map. We explore the use of numerical schemes (such as the Euler-Maruyama, a hybrid RK4, and an implicit scheme) to derive informed basis functions, leading to a parameter inference problem. We introduce a scalable algorithm to estimate the parameters by least squares, and we prove the convergence of the estimators as data size increases. 
    
    We test the ISALT on three non-globally Lipschitz SDEs: the 1D double-well potential, a 2D multiscale gradient system, and the 3D stochastic Lorenz equation with degenerate noise. Numerical results show that ISALT can tolerate time-step magnitudes larger than plain numerical schemes. It reaches optimal accuracy in reproducing the invariant measure when the time-step is medium-large.
\end{abstract}
\tableofcontents

\section{Introduction}

Efficient and accurate simulation of SDEs is important in many applications such as Monte Carlo sampling, data assimilation and predictive modeling (see e.g.,\cite{KP99,Maday2002,weinan2007heterogeneous,KMK03,legoll2010effective,CLMMT16}). In particular, ergodic systems often demand efficient and accurate simulation of both short-time dynamics and large-time statistics. Explicit schemes, while efficient and accurate for short-time, tend to miss the invariant measure in large-time simulations because of the accumulation of numerical error. In particular, for locally Lipschitz SDEs, they tend to be numerical unstable and may miss the invariant measure even for the small time-step (for example, the Euler-Maruyama scheme, because it destroys the Lyapunov structure \cite{roberts1996exponential,MSH02}) and require special treatments such as taming scheme under small time-step size \cite{Kloeden2012,Jentzen2015}. 
Implicit schemes, on the other hand, are numerically stable and can accurately simulate the invariant measure when the time-step is small. However, they are computationally inefficient due to the limited time-step size and the costly implicit step.  

We introduce ISALT, inference-based schemes adaptive to a large time-stepping statistical learning framework to construct schemes with large time-steps from data. When the data are generated from an implicit scheme, ISALT combines the advantages of both explicit and implicit schemes: it is as fast as an explicit scheme and is accurate as an implicit scheme in producing the invariant measure. The inference is done once for all and the inferred scheme can be used for general purpose simulations, either long trajectories or ensembles of trajectories with different initial distributions.

More specifically, we consider the large time-step approximation of the ergodic SDE with additive noise
\begin{equation}\label{eq:sde}
d\mX_t = f(\mX_t)dt + \sigma d\mB_t; \, 
\end{equation}
where the drift $f:\R^d\to\R^d$ is local-Lipschitz. Here $\mB$ is a standard m-dimensional Brownian motion with $m\leq d$, the diffusion matrix $\sigma\in \R^{d\times m}$ has linearly independent columns, and they represent a degenerate noise when $m<d$. Our goal is to design an explicit scheme with large time-stepping so that it can efficiently and accurately simulate both short-time dynamics
and large-time statistics such as invariant measures. 

We infer such explicit schemes with large time-stepping from offline data generated by an implicit scheme. Figure \ref{fig:schematic} shows the schematic plot of the procedure. The essential task is to approximate the infinite-dimensional discrete-time flow map. A major difficulty in a statistical learning approach is the curse of dimensionality (COD) when using generic basis functions. Our key contribution is to approximate the flow map by parametrization of numerical schemes, which provides informed basis functions, thus avoiding the COD by harnessing the rich information and structure in classical numerical schemes.  We also introduce a scalable algorithm to compute the maximal likelihood estimator by least squares, which converges and is asymptotic normal as the data size increases (see Theorem \ref{thm_convEst}). Furthermore, we show that the inferred scheme, when it is a parametrization of an explicit scheme, has 1-step strong order as the explicit scheme.  

\begin{figure}[htp!]
 \centering \vspace{-1mm}
    \includegraphics[width=0.7\textwidth]{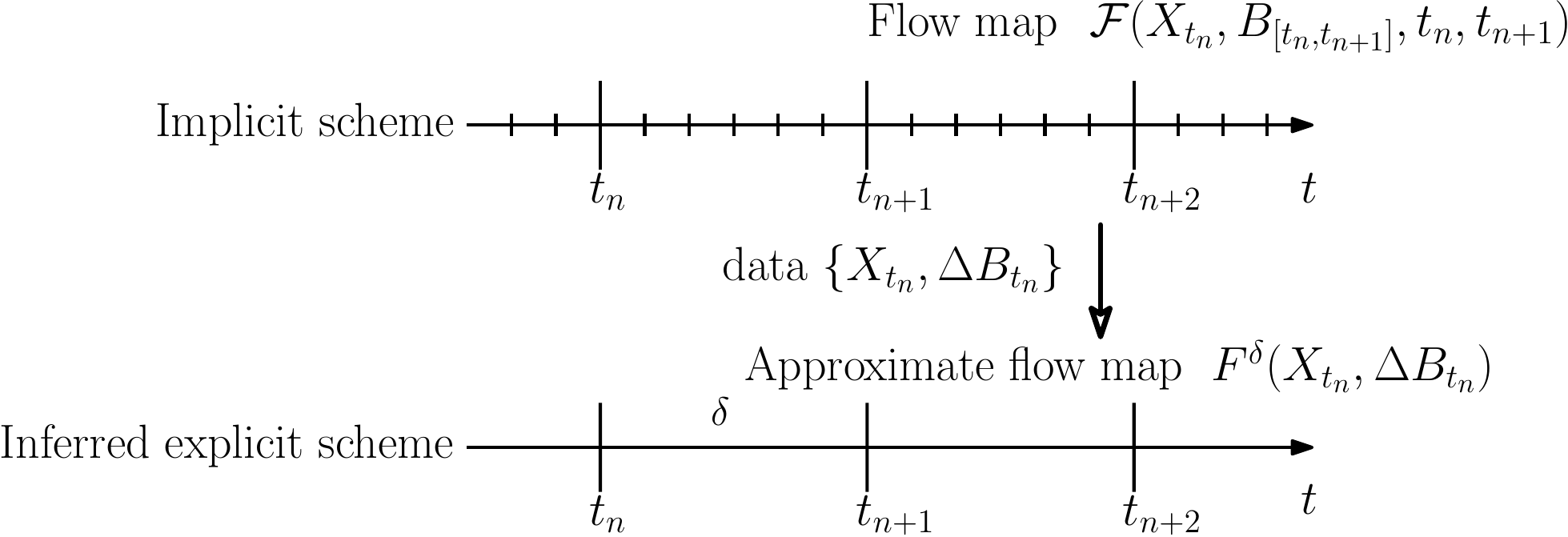}
 \vspace{-3mm}     \caption{Schematic plot of inferring explicit scheme with a large time-step.}
    \label{fig:schematic}
\end{figure}

 In this study, we focus on learning approximate flow maps that use only the increments of the Brownian motion each time interval (that is, the function $F^\deltat(X_{t_n}, \Delta B_{t_n})$ in Figure \ref{fig:schematic}). We explore the derivation of informed-basis functions from three types of classical numerical schemes: the  Euler-Maruyama (EM) \cite{KP99}, a hybrid RK4 (fourth-order Runge-Kutta)\cite{HP06}, and an implicit stochastic split backward Euler (SSBE) \cite{MSH02}, and we denote the inferred schemes by IS-EM, IS-RK4 and IS-SSBE. We test them on three non-globally Lipschitz SDEs: the 1D double-well potential, a 2D multiscale gradient system, and the 3D stochastic Lorenz equation with degenerate noise. Numerical results show that the inferred schemes can tolerate time-steps ten to hundreds times larger than plain numerical schemes, and it reaches optimal accuracy in reproducing the invariant measure at a medium large time-step (see Figures \ref{fig:TVDPDF_LangLocal1D}, \ref{fig:TVDPDF_GradCouple}, \ref{fig:TVDPDF_Lorenz}, and \ref{fig:TVDPDF_Lorenz_x3}). Overall, IS-RK4 produces the most accurate invariant measures in all examples, particularly when the dynamics is dominated by the drift (e.g., the Lorentz system) because the plain RK4 provides a higher order approximation to the drift.    

Discretization with large time-stepping for differential equations (SDEs, ODEs and PDEs) is a model reduction in time, part of the general problem of space-time model reduction (see e.g., \cite{weinan2007heterogeneous,KMK03,legoll2010effective,MH13,CL15,leiDatadrivenParameterization2016,Lu20Burgers,hudson2020coarse}). Since the large time-step prevents classical numerical approximations based on Taylor expansions, data-driven approaches have been the primary efforts and have witnessed many successes, including the time series approaches (see e.g., \cite{CL15,LLC17,LinLu20}) and deep learning methods that can efficiently solve high-dimensional PDEs and SDEs on rough space-time meshes (see e.g., \cite{han2018_SolvingHighdimensional,sirignano2018_DGMDeep,bar2019learning,Oosterlee2020,yang2020physics}), to name just a few. 
In these approaches, the discrete-time models account for the effects of the unresolved dynamics in an averaged fashion through inference, thus leading to computationally efficient models for the effective dynamics \cite{legoll2010effective,Legoll2019,CL15,LLC16}.
The contribution of our ISALT is to provide a simple yet effective approach to achieve large time-stepping by combining inference with classical numerical schemes. In particular, the explicit parametric form in ISALT clearly identifies the connection between classical numerical schemes and the models inferred from data. It provides a ground for further understanding the fundamental issues of data-based reduced models, such as quantification of the approximation and optimality of the  reduction in time or in space-time. 


The exposition of our study proceeds as follows. We first summarize the notations in Table 1. After introducing a flow map view for numerical schemes, we introduce in Section 2 the ISALT framework, that is, the procedure and algorithm for inferring schemes adaptive to the large time-step from data.  Section 3 presents the theoretical results on the convergence of the estimators. In Section 4, we test ISALT on the three typical non-globally Lipschitz SDEs. Section 5 concludes our main findings with an outlook of future research. 
\vspace{-2mm}
\begin{table}[H]
	\begin{center}
		\caption{Notations }
		\label{tab:notation-Infer}
		\begin{tabular}{ l  l }
		\toprule 
			Notation   &  Description \\  \hline
				$\mX_t$ and $\mB_t$			& true state process and original stochastic force  \\
				$f(\mX_t)$,\; $\sigma\in \mathbb{R}^{d\times m}$ with $m\le d$  & local-Lipschitz drift and diffusion matrix\\
				 $dt$  & time-step generating data \\
				$\deltat= \Gap \times dt$  & time-step for inferred scheme, $\Gap\in \{ 1, 2, 4, 10, 20, 40, 80,\ldots\}$ \\
				$t_i = i\deltat$ & discrete time instants of data \\
				$\{\mX_{t_0:t_N}^{(m)}, \mB_{t_0:t_N}^{(m)}\}_{m=1}^M$ & Data: $M$ independent paths of $\mX$ and $\mB$ at discrete-times \\ \hline
			$\mathcal{F}\left(\mX_{t_i},\, \mB_{[t_{i}, \, t_{i+1})}, t_{i}, t_{i+1}\right)$
			& true flow map representing  $(\mX_{t_{i+1}}-\mX_{t_i})/\deltat$\\
			${F}^{\delta}(\mX_{t_n},\Delta \mB_{t_n})$ & approximate flow map using only $\mX_{t_n}$ and $\Delta \mB_{t_n} = \mB_{t_{n+1}}-\mB_{t_{n}} $\\
			$\widetilde F^{\delta}\left(c^{\delta}, \mX_{t_n},\Delta \mB_{t_n} \right)$ & parametric approximate flow map \\ 
			$c^{\delta}=(c_0^{\delta},\dots,c_p^{\delta})=c_{0:p}^{\delta}$
			&  parameters to be estimated for the inferred scheme\\
			$\eta_n$ and $\sigma_{\eta}^{\delta}$		& iid $N(0, I_d)$ and a diagonal matrix, representing regression residual\\ \hline
		EM and IS-EM & plain Euler-Maruyama and inferred scheme (IS) parametrizing EM\\
		HRK4 and IS-RK4 & plain hybrid RK4 and inferred scheme parametrizing RK4 \\
		SSBE and IS-SSBE & split-step stochastic backward Euler and IS parametrizing it\\
			\bottomrule	
		\end{tabular}
	\end{center} \vspace{-2mm}
\end{table}

\section{Inference of explicit schemes from data}


Throughout this study, we assume that the SDE \eqref{eq:sde} is ergodic. Roughly speaking, the SDE is ergodic when  (i) there is a Lyapunov function $V:\R^d\to\R^+$ such that $\lim_{|x|\to\infty}V(x) =\infty$ and $\sup_{|x|>R}\mathcal{A}V(x) \to\infty$ as $R\to\infty$, where  $\mathcal{A}$ is the generator for \eqref{eq:sde} given by $\mathcal{A} g= \innerp{f,\nabla g} +\frac{1}{2} \sum_{i,j=1}^d [\sigma \sigma^\top]_{i,j} \partial_{x_i,x_j} g$ \cite{MSH02}; and (ii) $\mX$ satisfies a minority condition that ensures recurrence
\cite{MSH02,Khasminskii12}.

Our goal is to design a numerical scheme with a large time-step so that it can efficiently and accurately simulate both short-time dynamics 
and large-time statistics such as invariant measures. This is of particular interest for SDEs with non-globally Lipschitz drift, because explicit schemes such as Euler-Maruyama often blow up or miss the invariant measure even if they are stable  \cite{roberts1996exponential,MSH02} and implicit schemes are computationally costly while being accurate in large-time statistics. 

We obtain explicit schemes with large time-steps through inference from offline data generated by an implicit scheme. The key is to approximate the flow map by parametrization of numerical schemes, 
instead of using a generic basis, 
to avoid the curse of dimensionality in the statistical learning of the flow map.   
Toward the goal, we will first introduce the view that numerical schemes are approximations of the flow map, then we outline the framework of statistical learning of the flow map. 

\subsection{A flow map view of numerical schemes}

A numerical scheme aims to approximate the discrete-time flow map of the stochastic process. More precisely, for a time-step $\deltat> 0$, let $t_i= i\deltat$ and denote  $(\mX_{t_i}, i\geq 0)$ the process defined in \eqref{eq:sde} at discrete times. Based on the Markov property of $(\mX_t)$
a numerical scheme approximates the flow map
\begin{equation}\label{eq:flowmap}
\begin{aligned}
 \mX_{t_{i+1}}  - \mX_{t_{i}}  =\int_{t_i}^{t_{i+1}}f(\mX_s)ds +\int_{t_i}^{t_{i+1}} \sigma d\mB_s
 & =  \deltat\mcF(\mX_{t_i}, \mB_{[t_i,t_{i+1}]}, t_i, t_{i+1}) \\
 & \approx \deltat F^{\deltat} (\mX_{t_i}, \Delta \mB_{t_i}) 
 \end{aligned}
\end{equation}
where $\mcF$ is a functional depending on $\mX_{t_i}$, the continuous trajectory $ \mB_{[t_i,t_{i+1}]}$, $t_i$, and $t_{i+1}$.  The simplest schemes approximate the functional by a function $F^{\deltat} (\mX_{t_i}, \Delta \mB_{t_i})$ on $\R^{2d}$, in which one represents $ \mB_{[t_i,t_{i+1}]}$ by its increment on the interval $\Delta \mB_{t_i} =\mB_{t_{i+1}}-\mB_{t_i} \sim  \mathcal{N}(0, \deltat I_d)$. Among many such schemes (\cite{KP99,Hu96,MSH02,mao2007stochastic,leimkuhler2016molecular}), we consider three simple and representative examples: the explicit Euler-Maruyama scheme (EM) scheme \cite{KP99}, the hybrid RK4 (HRK4) \cite{HP06}, 
and the split-step stochastic backward Euler (SSBE) \cite{MSH02} 
\begin{equation}\label{eq:schemes}
\begin{aligned}
& \text{EM}   & \Xb_{n+1} &=  \Xb_n + f(\Xb_n) \deltat + \sigma \Delta \mB_n,   \\
& \text{HRK4}  & \Xb_{n+1} &=  \Xb_n + \phi_1 (\Xb_n,\sigma \Delta \mB_n) \deltat +  \sigma \Delta \mB_n,\\
& \text{SSBE} &  \Xb_{n+1}&= \Xb_* + \sigma \Delta \mB_n, \text{ with } \Xb_*=\Xb_n + f(\Xb_*)\deltat, \\
\end{aligned}
\end{equation}
where the term $ \phi_1$ is a standard RK4 step with the stochastic force treated as a constant input: 
\begin{align*} 
 \phi_1(\Xb_n,\sigma \Delta \mB_n): & = (k_1+2 k_2+2 k_3+k_4)/6, \text{ with } \\
    k_1 & = f(\Xb_n)+\sigma \Delta \mB_n/\deltat,\\
    k_2  &= f(\Xb_n + k_1\cdot \delta /2)+\sigma \Delta \mB_n/\deltat,\\
    k_3 &= f(\Xb_n + k_2\cdot \delta /2)+\sigma \Delta \mB_n/\deltat,\\
    k_4 &= f(\Xb_n + k_3\cdot \delta /2)+\sigma \Delta \mB_n/\deltat.
\end{align*}
Correspondingly, they approximate the flow map $\mcF(\mX_{t_i}, \mB_{[t_i,t_{i+1}]}, t_i, t_{i+1})$ by
\begin{equation}\label{eq:flowmap_schemes}
\begin{aligned}
& \text{EM}   & & F_{EM}^{\deltat} (\mX_{t_i}, \Delta \mB_{t_i})=  f(\mX_{t_i}) + \sigma \Delta \mB_{t_i}/\deltat,   \\
& \text{HRK4}   & & F_{RK4}^{\deltat} (\mX_{t_i}, \Delta \mB_{t_i})=  \phi_1(\mX_{t_i}, \sigma \Delta \mB_{t_i}) + \sigma \Delta \mB_{t_i}/\deltat,   \\
& \text{SSBE} &  &   F_{SSBE}^{\deltat} (\mX_{t_i}, \Delta \mB_{t_i})= (\mX_* - \mX_{t_i})/\deltat + \sigma \Delta \mB_{t_i}/\deltat, \text{ with } \mX_*=\mX_{t_i} +  f(\mX_*)\deltat.  \\
\end{aligned}
\end{equation}

For short time simulation, these schemes are of strong order 1, i.e., the discrete approximation converges to the true solution trajectory-wisely in probability at order $\delta^{1}$ as the time-step vanishes, since the noise is additive \cite{Rum82,KP99,Hu96,MSH02}.  %
For large time simulation aiming to approximate the invariant measure, the explicit schemes can be problematic for local Lipschitz drifts and degenerate noises, for instance, the EM scheme may destroy the Lyapunov structure and fail to be ergodic for any choice of time-step \cite[Lemma 6.3]{MSH02}. The implicit scheme SSBE, on the other hand, is ergodic and produces accurate invariant measure when the time-step is sufficiently small \cite[Section 6]{MSH02}. 

 In many applications, it is desirable to have an efficient numerical scheme being accurate in both short-time and large-time. A drawback of an implicit scheme is its inefficiency: it has to solve a fixed point problem in the implicit step, which is computationally costly and limits the time-step size. Taking advantage of implicit schemes, we use them to generate data and learn 
 efficient explicit schemes with large time-steps from the data.   

\subsection{Inference of a scheme from data} \label{sec:InferFramework}
We infer from data an explicit scheme that is accurate in both short-time dynamics and large-time statistics. It maintains the efficiency of explicit schemes while preserving the invariant measure as implicit schemes. The key idea is to learn an approximation of the flow map from data. To avoid the curse of dimensionality in the learning of the flow map, which is often high-dimensional and nonlinear, we derive parametric functions from the system and its numerical schemes. Roughly, the inference consists of four parts:
\begin{enumerate}
\item Generation of faithful data using an implicit scheme with a small time-step size; 
\item Derivation of a parametric form to approximate the flow map, by extracting basis functions from the system and its numerical approximations; 
\item Parameter estimation by maximal likelihood methods, which leads to a least squares problem when the parametric form is linear in the parameters;  
\item Model selection: by cross-validation and convergence criteria. 
\end{enumerate}

\paragraph{Data generation} We generate faithful data, consisting of trajectories of the process 
at discrete times $\{t_i=i\delta\}$, by an accurate implicit scheme. That is, we first solve the system by an implicit scheme with a small time-step $\Delta t< \delta$, then we downsample the solution at the discrete times. We also save the trajectory data of the stochastic force $(\mB_t)$. Denote these trajectories by 
\begin{equation}\label{eq:data}
\text{Data: } \{\mX_{t_0:t_N}^{(m)}, \mB_{t_0:t_N}^{(m)}\}_{m=1}^M,
\end{equation}
where $N$ denotes the number of observing time grids and $M$ denotes the number of independent trajectories.

The initial conditions $\{ \mX_{t_0}^{(m)} \}_{m=1}^M$ are samples from either a long trajectory, which represents the invariant measure, or an initial distribution that helps to explore the distribution of the process.  

\paragraph{Derivation of parametric form} The major difficulty in inference is the approximation of the flow map $\mcF(\mX_{t_i}, \mB_{[t_i,t_{i+1}]}, t_i, t_{i+1})$, which is an infinite-dimensional functional. When using a non-parametric approach with the generic dictionary or basis functions, one encounters the well-known curse-of-dimensionality (COD): the size of the dictionary or basis functions increases exponentially as the dimension increases. Recent efforts on overcoming the COD include selecting adaptive-to-data basis functions in a nonparametric fashion \cite{jiang2020modeling}, assuming a low-dimensional interaction between the components of the state variable in the spirit of particle interactions \cite{LZTM19}, or or deep learning methods that approximate high dimensional functions through compositions of simple functions \cite{han2018_SolvingHighdimensional,sirignano2018_DGMDeep,bar2019learning,Oosterlee2020,yang2020physics}.

We take a semi-parametric approach: we avoid the COD by deriving parametric functions from the full system and its numerical schemes, which provide rich information about the flow map. In particular, we aim for parametric functions depending linearly on the parameters, so that the parameters can be estimated by least squares and our algorithm is scalable .

We focus on approximating the flow map $\mcF(\mX_{t_i}, \mB_{[t_i,t_{i+1}]}, t_i, t_{i+1})$ by 
the simplest functions $F^{\deltat} (\mX_{t_i}, \Delta \mB_{t_i})$, in a parametric form
\begin{equation} \label{Fdelta}
 F^\deltat (c^\deltat; x,\xi) = \sum_{i=0}^p c_i^\deltat \phi_i(x,\xi),
\end{equation}
with $\xi$ having the same distribution as $\Delta \mB_{t_i}$. 
Here $\phi_i:\R^{2d}\to \R^d$ are basis functions to be extracted from numerical schemes (see Section \ref{sec:derivePara}), and $\{c_i^\delta\}$ are the parameters to be estimated from data. That is, with $(\Xb_n,\xi_n)$ corresponding to $(\mX_{t_n},\Delta \mB_{t_n}) $, we infer the following scheme 
\begin{equation} \label{RM_paraForm}
\Xb_{n+1} = \Xb_{n} + \deltat F^\deltat (c^\deltat; \Xb_n,\xi_n) +\deltat \sigma_\eta \eta_n =  \Xb_{n} + \deltat\sum_{i=0}^p c_i^\deltat \phi_i(\Xb_n,\xi_n) + \deltat\sigma_\eta^\deltat \eta_n,
\end{equation}
 where we add $\{\sigma_\eta^\deltat \eta_n\}$ to represent the residual of the regression.  For convenience, we assume that $\{\eta_n\}$ is a sequence of iid Gaussian $N(0,I_d)$ random variables and is independent of $\{\xi_n\}$, and $\sigma_\eta^\deltat$ is a diagonal matrix. 

In view of statistical learning, 
the function \eqref{Fdelta} approximates the flow map in the function space $\mathcal{H}=\mathrm{span}\{\phi_i(x,\xi)\}_{i=0}^p$, which is a subspace of $L^2(\R^{2d}, \mu\otimes \nu)$ with $\mu$ being the invariant measure of $\mX$ and $\nu\sim \mathcal{N}(0,\delta I_d)$ being the distribution of $\xi$ (which represents $\Delta \mB_{t_i}$).  We refer  $\{\phi_i(\helen{x},\xi)\}$ as basis functions and will extract them from numerical scheme (see Section  \ref{sec:derivePara}).    
 
 Here we focus on using only $\Delta \mB_{t_n}$, but one can use more sample points of the trajectory $\mB_{[t_n,t_{n+1}]}$ and extract terms from high-order approximations based on multiple stochastic integral \cite{Hu96}. We postpone this as future work.   
  
\paragraph{Parameter estimation}  We estimate the parameters by maximizing the likelihood for the model in \eqref{RM_paraForm} with the data $\{\mX_{t_0:t_N}^{(m)}, \mB_{t_0:t_N}^{(m)}\}_{m=1}^M$: 
\begin{align*}
l(c_{0:p}^\deltat)  
                           & = \frac{1}{M}\sum_{m=1}^M  l(\mX_{t_0:t_N}^{(m)}, \mB_{t_0:t_N}^{(m)} \mid c_{0:p}^\deltat), \text{ where } \\
l(\mX_{t_0:t_N}, \mB_{t_0:t_N} \mid c_{0:p}^\deltat) & =
\frac{1}{N} \sum_{k=1}^d\sum_{n=0}^{N-1} \left[ \frac{|\mX_{t_{n+1}}^k -\mX_{t_{n}}^k - \deltat  F_k^\deltat (c^\deltat,\mX_{t_i},\Delta \mB_{t_n})|^2 }{2\sigma_{k,\deltat}^2}  -\frac{1}{2} \log (2\pi (\sigma_{k,\deltat})^2) \right],
\end{align*}
where $F_k^\deltat$ is the $k$-th entry of the $\R^d$-valued function $F^\deltat$ defined in \eqref{Fdelta}: 
\[
F_k^\deltat (c^\deltat, \mX_{t_i},\Delta \mB_{t_n}) = \sum_{i=0}^p c_{i,k}^\deltat \phi_i^k(\mX_{t_n},\Delta \mB_{t_n}). 
\]

Noticing that the likelihood function is quadratic in the parameters $\big\{c_{i,k}^\delta\big\}_{i=0}^{p}$, we estimate them by least squares regression:
\begin{equation}\label{est_c}
\begin{aligned}
\widehat{c_{0:p,k}^{\delta,N,M}} & = (\widebar{A}^{N,M}_k)^+ \widebar{b}_k^{N,M},  \\
(\widehat{\sigma_{\eta,k}^{\deltat,N,M} })^2 & =\frac{1}{N} \sum_{n=0}^{N-1} |\mX_{t_{n+1}}^k -\mX_{t_{n}}^k - \deltat  F_k^\deltat (\widehat{c^{\deltat,N,M} },\mX_{t_i},\Delta \mB_{t_n})|^2, 
\end{aligned}
\end{equation}
where $A^{+}$ denotes the pseudo-inverse of $A$, and the normal matrix $\widebar{A}^{N,M}_k$ and vector $ \widebar{b}_k^{N,M}$ are given by 
\begin{equation}\label{LeatSquare_def}
\begin{aligned}
\widebar{A}^{N,M}_k(i,j) &=\frac{1}{MN} \sum_{m=1}^M\sum_{n=0}^{N-1} \phi_i^k(\mX_{t_n}^{k,(m)},\Delta \mB_{t_n}^{k,(m)}) \phi_j^k(\mX_{t_n}^{k,(m)},\Delta \mB_{t_n}^{k,(m)}),\quad i,j=0,\dots, p, \\
\widebar{b}_k^{N,M} (i) &= \frac{1}{MN} \sum_{m=1}^M\sum_{n=0}^{N-1} \frac{ \mX_{t_{n+1}}^{k,(m)}-\mX_{t_{n}}^{k,(m)}}{\deltat} \phi_i^k(\mX_{t_n}^{k,(m)},\Delta \mB_{t_n}^{k,(m)}) ,
\quad i=0,\dots,p.
\end{aligned}
\end{equation}
Here $\widehat{\sigma_{\eta,k}^{\deltat,N,M} }$, the square root of the regression's residuals, provide the diagonal entries of $\sigma_\eta^\deltat$.

The above least square regression is based on the assumption that the residual $\sigma^{\delta}\eta_n$ defined in \eqref{RM_paraForm} is  Gaussian with uncorrelated entries. The entry-wise regression aims to reflect the dynamical scale difference between entries. One may improve the approximation by considering correlated entries or other distributions for the residual.

\paragraph{Model selection}
The parametric form in Eq.\eqref{RM_paraForm} has many freedoms underdetermined, particularly when we have multiple options for the parametric form, along with possible overfitting and redundancy in these options. We select the estimated scheme by the following criteria: 
\begin{itemize}
\item Cross validation: the estimated scheme should be stable and can reproduce the distribution of the process, particularly the main dynamical-statistical properties. We will consider the marginal invariant densities and temporal correlations:  
\begin{equation} \label{eq:stats} 
\begin{aligned} 
\text{Invariant density of $(\mX_t^k)$: } & p(z)dz = \E  [\mathbf{1}_{ (z, z+ dz)}(\mX_{t_n}^{k})]  \approx \frac{1}{N M}\sum_{m,n=1}^{M,N} \mathbf{1}_{ (z, z+ dz)}(\mX_{t_n}^{k,(m)}), \\ 
\text{Temporal correlations: } & C_k(h)= \E  [\mX_{t_n+h}^{k} \mX_{t_n}^{k}]   \approx \frac{1}{N M}\sum_{m,n=1}^{M,N} \mX_{t_n+h}^{k,(m)} \mX_{t_n}^{k,(m)} 
\end{aligned}
 \end{equation}
 for $k=1,\ldots,d$.

\item Convergence of the estimators. If the model is perfect and the data are either independent trajectories or a long trajectory from an ergodic measure, the estimators should converge to the true values when the data size increases (see Theorem~\ref{thm:const_perfect_model}). While our parametric model is not perfect, the estimators should also converge when the data size increases (see Theorem~\ref{thm_convEst}) and highly oscillatory estimators indicate large misfits between the proposed model and data. 
\end{itemize}

\subsection{Parametrization of numerical schemes} \label{sec:derivePara}

We derive parametric forms to approximate the flow map from numerical schemes. The numerical schemes provide informed basis functions for inference because of their error-controlled approximations to the flow map $\mcF(\mX_{t_i}, \mB_{[t_i,t_{i+1}]},t_i,t_{i+1})$ in \eqref{eq:flowmap}. These basis functions can either be simply the terms in an explicit scheme or terms approximating the implicit schemes. One may view this approach as a parametrization of numerical schemes. 

We focus on using only  $\Delta \mB_{t_i}$, the increment of $\mB_{[t_i,t_{i+1}]}$, and seek parametric functions 
$F^{\deltat} (c^{\deltat},\mX_{t_i}, \Delta \mB_{t_i})$ (as in \eqref{Fdelta}) to approximate the flow map. This constraint has two advantages: first, it makes the inferred-scheme computationally efficient, because the inferred scheme will generate only two random numbers ($\xi_i, \eta_i$ in \eqref{RM_paraForm}) in each time step to represent the stochastic forces; second, it significantly reduces the function space of inference, from a functional depending on the path $\mB_{[t_i,t_{i+1}]}$ to a function depending only on the increment.  By starting from this simple setting, we hope to provide insight on the future design of schemes using multi-point noise by parametrizing high-order stochastic schemes (see e.g.\cite{Hu96,KP99,jentzen2010taylor}).  


The flow maps \eqref{eq:flowmap_schemes} of the numerical schemes in \eqref{eq:schemes} provide three representative candidates for a parametric function $ \widetilde F^{\deltat} (c^{\deltat},\mX_{t_i}, \Delta \mB_{t_i})$. The EM is an explicit one-step scheme, the RK4 is an explicit multi-step scheme, and the SSBE is an implicit one-step scheme. Linearly parametrizing them or their Ito-Taylor expansions, i.e., adding coefficients to the terms,  we obtain parametric flow maps:
\begin{equation}\label{eq:paraFormEuler}
\begin{aligned}
& \text{EM}   &  \widetilde F_{EM}^{\deltat} (c^{\deltat};\mX_{t_i}, \Delta \mB_{t_i}) &=  c_0^{\deltat} \mX_{t_i} + c_1^{\deltat} f(\mX_{t_i}) + c_2^{\deltat} \sigma \Delta \mB_{t_i}/\deltat,   \\
& \text{HRK4}   &  \widetilde F_{RK4}^{\deltat} (c^{\deltat};\mX_{t_i}, \Delta \mB_{t_i}) &=  c_0^{\deltat} \mX_{t_i} + c_1^{\deltat} \phi_1(\mX_{t_i},  \sigma \Delta \mB_{t_i}) + c_2^{\deltat} \sigma \Delta \mB_{t_i}/\deltat,   \\
& \text{SSBE} &     \widetilde F_{SSBE}^{\deltat} (c^{\deltat};\mX_{t_i}, \Delta \mB_{t_i}) &= c_0^{\deltat} \mX_{t_i} + c_1^{\deltat} \phi_1^{SSBE}(\mX_{t_i}) + c_2^{\deltat} \sigma \Delta \mB_{t_i}/\deltat,
\end{aligned}
\end{equation}
where the function 
 $\phi_1^{SSBE}$ is given by 
\begin{equation}\label{phi1_BEs}
\begin{aligned}
\phi_1^{SSBE}(\mX_{t_i}) & =  (I_d - \deltat  \nabla f(\mX_{t_i})) ^{-1} f(\mX_{t_i}). 
\end{aligned}
\end{equation}
These terms are derived as follows. 
\begin{itemize}
\item The parametric flow map $\widetilde F_{EM}^{\deltat} (c^{\deltat};\mX_{t_i}, \Delta \mB_{t_i}) $ and $\widetilde F_{RK4}^{\deltat} (c^{\deltat};\mX_{t_i}, \Delta \mB_{t_i}) $ come simply by adding coefficients to each term in $F_{EM}^{\deltat} $ and $F_{RK4}^{\deltat} $ of the Euler and RK4 schemes  in \eqref{eq:flowmap_schemes}.  

\item We introduced an extra linear term $c_0^{\deltat} \mX_{t_i} $. When $f$ is nonlinear, it serves as a linear basis function, and it helps to data-adaptively adjust the linear stability of the inferred scheme. 

\item The parametric flow maps $\widetilde F_{SSBE}^{\deltat} (c^{\deltat};\mX_{t_i}, \Delta \mB_{t_i}) $ 
comes from parametrizing the terms in an approximation of $F_{SSBE}^{\deltat} (\mX_{t_i}, \Delta \mB_{t_i})$ in \eqref{eq:flowmap_schemes}. More precisely, by the mean-value theorem, there exists a state $\widetilde{\mX}_{t_i}$ depending on $\mX_*$ and $\mX_{t_i}$ such that
\begin{equation}\label{eq:res_SSBE}
\begin{aligned}
f(\mX_*) &=  f(\mX_{t_i}) + \nabla  f(\widetilde{\mX}_{t_i}) (\mX_*-\mX_{t_i}) \\
& = f(\mX_{t_i}) + \nabla  f(\mX_{t_i}) (\mX_*-\mX_{t_i}) +R(\mX_*,\mX_{t_i},\nabla f),
\end{aligned}
\end{equation}
where $R(\mX_*,\mX_{t_i},\nabla f)= [\nabla  f(\widetilde{\mX}_{t_i})-\nabla  f(\mX_{t_i})] (\mX_*-\mX_{t_i}) $.
Then, by the definition of $\mX_*$ in the SSBE in \eqref{eq:flowmap_schemes}, we have  
 \begin{align*}
& \mX_*=\mX_{t_i}+ \deltat f(\mX_*) =  \mX_{t_i} + \deltat [ f(\mX_{t_i}) +  \nabla f(\mX_{t_i}) (\mX_{*}-\mX_{t_i})] +R(\mX_*,\mX_{t_i},\nabla f) \\
& \Rightarrow  (\mX_*-\mX_{t_i})        =  (I_d - \deltat  \nabla f(\mX_{t_i})) ^{-1} \deltat f(\mX_{t_i}) +R(\mX_*,\mX_{t_i},\nabla f) .
\end{align*}
Thus, we have
\[
F_{SSBE}^{\deltat} (\mX_{t_i}, \Delta \mB_{t_i})= (I_d - \deltat  \nabla f(\mX_{t_i})) ^{-1}  f(\mX_{t_i}) + \sigma \Delta \mB_{t_i}/\deltat + R(\mX_*,\mX_{t_i},\nabla f). 
\]
Assuming that $R(\mX_*,\mX_{t_i},\nabla f)$ is negligible,  parametrizing the other terms, and adding $c_0^{\deltat} \mX_{t_i} $, we obtain $\widetilde F_{SSBE}^{\deltat}$ with $\phi_1^{SSBE}$ above. Note that when $f$ is globally Lipschitz (thus $|\nabla f|$ is bounded above), we have $\E[ |R(\mX_*,\mX_{t_i},\nabla f)|] \leq C \E[|\mX_*-\mX_{t_i}|^2]$, i.e., $R(\mX_*,\mX_{t_i},\nabla f)$ is an order smaller than $\mX_*-\mX_{t_i}$. However, when $f$ is non-globally Lipschitz (thus $|\nabla f|$ is unbounded ), $R(\mX_*,\mX_{t_i},\nabla f)$ may be non-negligible and require additional terms to account for its effect. 
\end{itemize}

\bigskip

Putting the parametric flow maps in the form in \eqref{RM_paraForm}, the corresponding inferred schemes (IS) with these parametrized flow maps in \eqref{eq:paraFormEuler} are 
\begin{equation}\label{eq:IS_Euler}
\begin{aligned}
& \text{IS-EM}   &  (\mX_{t_{i+1}} - \mX_{t_i})/\deltat  &=  c_0^{\deltat} \mX_{t_i} + c_1^{\deltat} f(\mX_{t_i}) + c_2^{\deltat} \sigma \Delta \mB_{t_i}/\deltat + \sigma_\eta \eta_i,   \\
& \text{IS-RK4: }   & (\mX_{t_{i+1}} - \mX_{t_i})/\deltat &=c_0^{\deltat}  \mX_{t_i} + c_1^{\deltat} \phi_1 (\mX_{t_i},\sigma \Delta \mB_{t_i})+
c_2^{\deltat} \sigma \Delta \mB_{t_i}/\deltat+\sigma_{\eta}\eta_{i} \\
& \text{IS-SSBE} &    (\mX_{t_{i+1}} - \mX_{t_i})/\deltat &= c_0^{\deltat} \mX_{t_i} + c_1^{\deltat} \phi_1^{SSBE}(\mX_{t_i}) + c_2^{\deltat} \sigma \Delta \mB_{t_i}/\deltat + \sigma_\eta \eta_i. 
\end{aligned}
\end{equation}

We point out that there are many other options for the parametric form. These three to-be-inferred schemes are typical: IS-EM and IS-RK4 are explicit schemes, and they will improve the statistical accuracy of the plain EM or RK4 by design (see Section \ref{sec:convImperfect}). IS-RK4 is based on a multi-step scheme which provides a high-order approximation of the drift, so it is likely to perform better than IS-EM when it is stable. The IS-SSBE comes from an implicit scheme, and is likely to inherit the stability. 


\subsection{Algorithm}

The following algorithm summarizes that above procedure for the inference of a scheme.  
\begin{algorithm}[H]
{\small
\caption{Inference-based reduced model with memory: detailed algorithm}\label{alg:main}
\begin{algorithmic}[1]
\Require{Full model; a high fidelity solver preserving the invariant measure.  }
\Ensure{Estimated parametric scheme}
\State  Generate data: solve the system with the high fidelity solver, which has a small time-step $dt$; down sample to get time series with $\deltat= \mathrm{Gap}\times dt$.  Denote the data, consisting of $M$ independent trajectories on $[0,N\deltat]$, by $\{\mX_{t_0:t_N}^{(m)}, \mB_{t_0:t_N}^{(m)}\}_{m=1}^M$ with $t_i= i\deltat$. 
\State Pick a parametric form approximating the flow map \eqref{eq:flowmap} as in \eqref{Fdelta}--\eqref{RM_paraForm}.
\State Estimate parameters $c_{0:p}^\deltat$ and $\sigma_\eta$ as in \eqref{est_c}. 
\State Model selection: run the inferred scheme for cross-validation, and test the consistency of the estimators.
\end{algorithmic}
}
\end{algorithm}

 \section{Convergence of estimators}\label{sec:conv}
 We consider the convergence of the estimators in sample size in two settings: perfect model and imperfect model. The perfect model setting aims to validate our algorithm, in the sense that the algorithm can yield consistent and asymptotically normal estimators. 
 The imperfect model setting is what we have in practice, and we show that our estimator converges to the (optimal) projection. In particular, we show that an inferred-scheme improves the statistical accuracy of its explicit counterpart.  
 
For simplicity of notation, we assume that $d=1$ throughout this section. But the results also hold true entry-wisely for the system with $d>1$. 

 \subsection{Convergence of estimator for perfect model}

We denote the expectation of $\widebar{A}^{N,M}$ and $\widebar{b}^{N,M}$ in \eqref{LeatSquare_def} by $A$ and $b$: 
\begin{equation}\label{Aexpectation}
\begin{aligned}
A= \E[\widebar{A}^{N,M}] &= \frac{1}{N} \sum_{n=0}^{N-1} \left( \E\left[ \langle \phi_i(\mX_{t_n}^{(m)}, \Delta \mB_{t_n}^{(m)}), \phi_j(\mX_{t_n}^{(m)},\Delta \mB_{t_n}^{(m)})\rangle_{\R^d} \right] \right)_{i,j}, \\
b = \E[\widebar{b}^{N,M}] & = \frac{1}{N} \sum_{n=0}^{N-1}(\E[  \langle \frac{ \mX_{t_{n+1}}^{(m)}-\mX_{t_{n}}^{(m)}}{\deltat}, \phi_i(\mX_{t_n}^{(m)},\Delta \mB_{t_n}^{(m)}) \rangle_{\R^d} ])_i.
\end{aligned}
\end{equation}
Here the expectation is with respect to the distribution filtration generated by the initial distribution and the Brownian motion.  
\begin{assumption}\label{Assum0}
(a) Suppose that the data $\{\mX_{t_0:t_N}^{(m)}, \mB_{t_0:t_N}^{(m)}\}_{m=1}^M$ are independent trajectories of the system  \eqref{RM_paraForm} with $\{\mX_{t_0}^{(m)}\}_{m=1}^M$ sampled from the ergodic measure of $\mX$. (b) Suppose that the normal matrix $\widebar{A}^{N,M}$  in \eqref{LeatSquare_def} and its expectation in \eqref{Aexpectation}
 are invertible. (c) Suppose that the flow map $\mcF^\deltat$ in \eqref{eq:flowmap} is square integrable. 
\end{assumption} 
  \begin{theorem}[Consistency and asymptotic normality for perfect model]\label{thm:const_perfect_model}
Under Assumption {\rm \ref{Assum0}}, 
the estimator in \eqref{est_c} converges to $c^{\deltat}$ (the true parameter value) almost surely, and  is asymptotically normal, when either $M\to \infty$ or $N\to \infty$:  
\begin{equation}
 \begin{aligned}
& \sqrt{M} (\widehat{c^{\deltat, N,M}} - c^{\deltat})  \xrightarrow[]{d}   \mathcal{N}(0,\frac{1}{N}\sigma_\eta^2 A),  \\
 & \sqrt{N} (\widehat{c^{\deltat, N,M}} - c^{\deltat})  \xrightarrow[]{d}  \mathcal{N}(0,\frac{1}{M}\sigma_\eta^2 A).
 \end{aligned}
 \end{equation}
\end{theorem}
\begin{proof}
By definition of $\widebar{b}^{N,M}$ in \eqref{LeatSquare_def} and the equation \eqref{RM_paraForm}, we have
\begin{align*}
\widebar{b}^{N,M}(i) 
& = \frac{1}{MN} \sum_{m=1}^M\sum_{n=0}^{N-1}  \langle \sum_{j=0}^p c^\deltat_j \phi_j (\mX_{t_n}^{(m)},\Delta \mB_{t_n}^{(m)})+ \sigma_\eta \eta_n^{(m)}, \phi_i(\mX_{t_n}^{(m)} ,\Delta \mB_{t_n}^{(m)})\rangle_{\R^d}  \\
& = \left( \widebar{A}^{N,M}c^\deltat\right)(i) +  \widebar{S}^{N,M} , 
\end{align*}
where in the second equality we used the definition of $\widebar{A}^{N,M}$  in \eqref{LeatSquare_def}, and we denote
\[
 \widebar{S}^{N,M} = \frac{1}{M} \sum_{m=1}^M S^{N,(m)}, \, \text{ with } S^{N,(m)}= \frac{1}{N} \sum_{n=0}^{N-1}  \langle \sigma_\eta \eta_n^{(m)}, \phi_i(\mX_{t_n}^{(m)} ,\Delta \mB_{t_n}^{(m)})\rangle_{\R^d}. 
\]
Note that $\eta_n$ is standard Gaussian and is independent of $\mB_{t_n}$ and $\mX_{t_n}$. Then, 
 $S^{N,(m)}$
has mean zero and its covariance is 
\[
\mathrm{Cov}(S^{N,(m)}) =  \sigma_\eta^2  \frac{1}{N^2}  \sum_{n,n'=0}^{N-1}  \E \left[ \langle \eta_n^{(m)}, \phi_i(\mX_{t_n}^{(m)} ,\Delta \mB_{t_n}^{(m)})\rangle_{\R^d}  \langle \eta_{n'}^{(m)}, \phi_i(\mX_{t_{n'}}^{(m)} ,\Delta \mB_{t_{n'}}^{(m)})\rangle_{\R^d} \right] =  \frac{1}{N}\sigma_\eta^2 A. 
\]
Thus, when $M\to \infty$, we have by the central limit theorem, 
\begin{equation}\label{S_AN_M}
\sqrt{M} \frac{1}{M} \sum_{m=1}^M S^{N,(m)}  \xrightarrow[]{d}   \mathcal{N}(0, \frac{1}{N}\sigma_\eta^2 A); 
\end{equation}
Furthermore,  $S^{N,(m)}$ is a martingale with respect to the filtration generated by $\{\mX_{t_n}, \mB_{t_n}, \eta_n\}$, and 
when $N\to \infty$, we have by martingale central limit theorem \cite[Theorem 3.2]{hall2014martingale}
\begin{equation} \label{S_AN_N}
\sqrt{N} \frac{1}{M} \sum_{m=1}^M S^{N,(m)} \xrightarrow[]{d}   \mathcal{N}(0, \frac{1}{M}\sigma_\eta^2 A ). 
\end{equation}

We show first that, when $M\to \infty$ and for each fixed $N$,  the estimator is consistent and asymptotically normal. Note that by the strong  Law of Large Numbers, $ \widebar{A}^{N,M}\to A$ and $\widebar{b}^{N,M}\to b$ a.s.~as $M\to \infty$. Thus, $ (\widebar{A}^{N,M})^{-1} \to A^{-1}$ almost surely (using the fact that $A^{-1}-B^{-1} = A^{-1}(B-A)B^{-1}$, see \cite[page 22]{LMT20}). Then,  $\widehat{c^{\deltat, N,M}}=(\widebar{A}^{N,M})^{-1} \widebar{b}^{N,M} \to A^{-1}b$ almost surely, i.e. the estimator is consistent. Combining \eqref{S_AN_M} and the almost sure convergence of $(\widebar{A}^{N,M})^{-1} $, we obtain the asymptotic normality by noticing that
\[
\widehat{c^{\deltat, N,M}}=(\widebar{A}^{N,M})^{-1} \widebar{b}^{N,M} = c^\deltat + (\widebar{A}^{N,M})^{-1} \widebar{S}^{N, M} . 
\]

When $N\to\infty$ and $M$ fixed, we obtain $ \widebar{A}^{N,M}\to A$ and $\widebar{b}^{N,M}\to b$ a.s.~by the ergodicity of the process. The consistency and asymptotic normality follows similarly by using \eqref{S_AN_N}. 
\end{proof}

 \subsection{Convergence of estimator for imperfect model} \label{sec:convImperfect}
In practice, the model is imperfect in our inferred scheme because we can rarely parametrize the flow map exactly. We show next that for an imperfect proposed model, the estimator converges to the projected coefficient of the flow map onto the function space spanned by the proposed basis in the ambient $L^2$ space. 
Furthermore, we show that the inferred scheme improves the statistical accuracy of the explicit scheme that it parametrizes. 

\begin{assumption}\label{Assum1}
(a) Suppose that the data $\{\mX_{t_0:t_N}^{(m)}, \mB_{t_0:t_N}^{(m)}\}_{m=1}^M$ are independent trajectories of the system \eqref{eq:sde} with $\{\mX_{t_0}^{(m)}\}_{m=1}^M$ sampled from the ergodic measure of $\mX$. (b) Suppose that the normal matrix $\widebar{A}^{N,M}$  in \eqref{LeatSquare_def} and its expectation in \eqref{Aexpectation}
 are invertible. (c) Suppose that the flow map $\mcF^\deltat$ in \eqref{eq:flowmap} is square integrable. 
\end{assumption}
The invertibility of the normal matrices $\widebar{A}^{N,M}$ and $A$ is crucial for our theory, and they lead to constraints on the basis functions. In practice, we can use it to guide the selection of basis functions and we recommend using pseudo-inverse and regularization when the normal matrix is close to singular.   

With the notation $A$ and $b$ in \eqref{Aexpectation}, and assuming that $A$ is invertible,  we define
\begin{equation}\label{eq:c_proj}
c^{\deltat,\mathrm{proj}} := A^{-1} b. 
\end{equation}
The following lemma shows that  $c^{\deltat,\mathrm{proj}} $ is the projection coefficients of the flow map $\mcF^\deltat$. 

\begin{lemma} \label{lemma:c_proj}
Under Assumption {\rm \ref{Assum1}}, the vector  $c^{\deltat,\mathrm{proj}} $ in \eqref{eq:c_proj} is the projection coefficient of the flow map $\mcF^\deltat$ in \eqref{eq:flowmap} onto the space  $\mathrm{span}\{\phi_i\}_{i=0}^p$ in $L^2(\R^{d}\times \Omega^\deltat, \mu\otimes \nu)$ with  $\mu$ being the invariant measure of $\mX$ and $(\Omega^\deltat, \mcB,\nu)$ being the canonical probability space for the Brownian motion $(\mB_t, t\in [0,\deltat])$.
\end{lemma}
\begin{proof} Note that $\mcF^\deltat_{t_n} =  \frac{\mX_{t_{n+1}}-\mX_{t_{n}}}{\deltat}$. Denote $\mcF^{\deltat, m}_{t_n} =  \frac{\mX_{t_{n+1}}^{(m)}-\mX_{t_{n}}^{(m)}}{\deltat}$. By the definition of $b$ in \eqref{Aexpectation}, we have
\begin{align*}
b(i)
& = \frac{1}{N} \sum_{n=0}^{N-1} \E \langle\mcF^{\deltat,m}_{t_n}, \phi_i(\mX_{t_n}^{(m)} ,\Delta \mB_{t_n}^{(m)})\rangle_{\R^d} =  \E \langle\mcF^{\deltat}_{t_n}, \phi_i(\mX_{t_n} ,\Delta \mB_{t_n})\rangle_{\R^d}, 
\end{align*}
where the second equality follows from that $( \mX_{t_n}, \Delta \mB_{t_n})$ is stationary (so does $\mcF^{\deltat}_{t_n}$). 

Denote by $c = (c_0,c_1,\ldots, c_p)^\top$ the projection coefficients of $\mcF^\deltat_{t_n}$ to $\mathrm{span}\{\phi_i\}_{i=0}^p$, and write $\mcF^{\deltat}_{t_n} = \sum_{i=0}^p c_i \phi_i +  \mcF $ with $\mcF$ satisfying $\E[\innerp{\mcF, \phi_i}_{\R^d} ] =0$ for each $i=0,1,\ldots, p$. Then 
\[
\E[\innerp{ \mcF^{\deltat}_{t_n}, \phi_i}_{\R^d} ]= \sum_{j=0}^p c_j \E[\innerp{\phi_j, \phi_i}_{\R^d} ] = (A c)(i). 
\]
Combining the above two equations, we obtain that  $c^{\deltat,\mathrm{proj}}  = A^{-1}b = c$. 
\end{proof}

We remark that because $\mcF^\deltat(\mX_{t_i}, \mB_{[t_i,t_{i+1}]}, t_i, t_{i+1})$ is a functional depending on the trajectory $\mB_{[t_i,t_{i+1}]}$,  the function space of projection, $L^2(\R^{d}\times \Omega^\deltat, \mu\otimes \nu)$, has an infinite dimensional state space for $(\mX_{t_i}, \mB_{[t_i,t_{i+1}]})$. When $\mcF^\deltat_{t_n}$ depends only on $(\mX_{t_n}, \Delta \mB_{t_n})$ (for instance, in the case of perfect model discussed in the previous section), the state space becomes finite dimensional and the function space is simplified to $L^2(\R^{2d}, \mu\otimes \nu)$ with $\nu\sim \mathcal{N}(0,\deltat I_d)$. 

 \begin{theorem}[Convergence of the estimator] \label{thm_convEst}
In addition to Assumption {\rm \ref{Assum1}}, assume that $\E[|\mcF^\deltat_{t_0}|^4] <\infty$ and $\E[|\phi_i(\mX_{t_0}, \Delta \mB_{t_0}) |^4] <\infty$ for each $i=0,\dots,p$.  Then, we have 
\begin{itemize} 
\item when $M\to \infty$ and $N$ fixed, the estimator in \eqref{est_c} converges to the projection coefficients $c^{\deltat,\mathrm{proj}}$ in  \eqref{eq:c_proj} a.s.~and is asymptotically normal:
\begin{equation}\label{AN_inM}
 \begin{aligned}
& \sqrt{M} (\widehat{c^{\deltat, N,M}} - c^{\deltat,\mathrm{proj}})  \xrightarrow[]{d}  \mathcal{N}(0,A^{-1}\Sigma^N (A^{-1})^\top),  
 \end{aligned}
 \end{equation}
 where the matrix $\Sigma^N$ is the covariance of  
  \[ 
\widetilde{b}^{N,m}(i) = \frac{1}{N}\sum_{n=0}^{N-1} b^{n,m},\, \text{  with }  b^{n,m}= \langle \frac{ \mX_{t_{n+1}}^{(m)}-\mX_{t_{n}}^{(m)}}{\deltat}, \phi_i(\mX_{t_n}^{(m)},\Delta \mB_{t_n}^{(m)}) \rangle_{\R^d}  . 
\]
\item when $N\to \infty$ and $M$ fixed, the estimator in \eqref{est_c} also converges and is asymptotically normal
 \begin{equation}\label{AN_inN}
 \sqrt{N} (\widehat{c^{\deltat, N,M}} - c^{\deltat,\mathrm{proj}})  \rightarrow \mathcal{N}(0,\frac{1}{M}A^{-1}\Sigma (A^{-1})^\top),
 \end{equation}
with $\Sigma  = \lim_{N\to \infty}N\Sigma^N$, provided that for each $m$, 
 \begin{align*}
 & \sum_{n=0}^\infty\E[b^{n,m} \E[b^{k,m} | \mcM_0] ] \, \text{converges for each $k\geq 0$ and}  \\
 & \lim_{N\to \infty} \sum_{k=K}^\infty \E[b^{k,m} \E[b^{N,m} | \mcM_0] ] =0 \text{ uniformly in $K$},  
 \end{align*}
where $\mcM_0$ denotes the filtration generated by the extended stationary process up to time $t_0$. 
 \end{itemize}
\end{theorem}

\begin{proof} When $M\to \infty$,  by the strong Law of Larger Numbers, we have $ \widebar{A}^{N,M}\to A$ and $\widebar{b}^{N,M}\to b$ a.s.~as $M\to \infty$. Thus,  $\widehat{c^{\deltat, N,M}}=(\widebar{A}^{N,M})^{-1} \widebar{b}^{N,M}  \to  A^{-1}b= c^{\deltat,\mathrm{proj}}$ a.s. according to Lemma~\ref{lemma:c_proj}. To prove the asymptotic normality, note that for each $m$, the random vector $\widetilde{b}^{N,m}$ with entries
\[ 
\widetilde{b}^{N,m}(i) = \frac{1}{N}\sum_{n=0}^{N-1} \langle \mcF^{\deltat,m}_{t_n}, \phi_i(\mX_{t_n}^{(m)},\Delta \mB_{t_n}^{(m)}) \rangle_{\R^d}
\]
has mean $\E[\widetilde{b}^{N,m}] =b$ and covariance  $\Sigma^N$ with entries $\Sigma^N_{i,j} = \E[\widetilde{b}^{N,m}(i) \widetilde{b}^{N,m}(j)] - b(i)b(j)$. 
 Here the covariance exists because 
\begin{align*}
 \E[\widetilde{b}^{N,m}(i) \widetilde{b}^{N,m}(j)] &\leq  \max_i \E[|\widetilde{b}^{N,m}(i)|^2] \leq \max_i \E [| \langle \mcF^{\deltat,m}_{t_n}, \phi_i(\mX_{t_n}^{(m)},\Delta \mB_{t_n}^{(m)}) \rangle_{\R^d}|^2] \\
 & \leq ( \E[|\mcF^\deltat_{t_0}|^4])^{1/2} \max_i  (\E[|\phi_i(\mX_{t_0}, \Delta \mB_{t_0}) |^4)^{1/2}. 
\end{align*}
Then, $ \widebar{b}^{N,M} $ is the average 
of $M$ iid samples $\{\widetilde{b}^{N,m}\} $, each of which has covariance $\Sigma^N$. Hence, by the central limit theorem, we have 
\begin{align*}
\sqrt{M}(\widebar{b}^{N,M} - b) &   \xrightarrow[]{d}   \mathcal{N}(0,\Sigma^N). 
\end{align*}
Combining with the facts that $ \widebar{A}^{N,M}\to A$ a.s.~and that these matrices are invertible, we obtain
 \eqref{AN_inM}.  

\bigskip
When $N\to \infty$, we obtain the convergence and asymptotic normality by ergodicity. First, by ergodicity, we have $ \widebar{A}^{N,M}\to A$ and $\widebar{b}^{N,M}\to b$ a.s.~as $M\to \infty$. Thus,  $\widehat{c^{\deltat, N,M}}=(\widebar{A}^{N,M})^{-1} \widebar{b}^{N,M}  \to  A^{-1}b= c^{\deltat,\mathrm{proj}}$ almost surely. Next, to prove the asymptotic normality, note that by the central limit theorem for stationary processes \cite[Theorem 1]{heyde1974central}, we have $\Sigma  = \lim_{N\to \infty}N\Sigma^N$ and  
\begin{align*}
\sqrt{N}(\widebar{b}^{N,M} - b) &   \xrightarrow[]{d}   \mathcal{N}(0,\frac{1}{M}\Sigma).
 \end{align*}
 Then we obtain \eqref{AN_inN} by noting that  $ \widebar{A}^{N,M}\to A$ a.s. as above. 
\end{proof}

\bigskip
We show next that the parametrization by inference will lead to improvement to an explicit scheme, in the sense that the residual of the inferred scheme is not larger than the explicit scheme. In other words, the 1-step discretization error of the inferred scheme is not larger than the explicit scheme. 
 \begin{theorem}[Order of residual] \label{thm:order_res}
Assume that the parametric form in \eqref{RM_paraForm} comes from an explicit scheme, such as {\rm IS-EM} or {\rm IS-RK4} in \eqref{eq:IS_Euler} from its explicit scheme in \eqref{eq:flowmap_schemes}.  Then, under Assumption {\rm \ref{Assum1}}, the inferred scheme's residual is smaller than the explicit scheme's. Specifically, the residual in \eqref{est_c} satisfies
\begin{equation}\label{eq:order_res}
\E \widehat{(\sigma^{N,M})^2}  \leq \frac{2}{d} \E\Big[| \frac{\mX_{t_{n+1}}-\mX_{t_{n}}}{\deltat}  - F^\deltat(\mX_{t_n}, \Delta \mB_{t_n} )|^2 \Big], 
\end{equation}
where $F^\deltat(\mX_{t_n}, \Delta \mB_{t_n} )$ denotes the flow map of the explicit scheme, such as $ F^\deltat_{EM}$ or $ F^\deltat_{RK4}$ in \eqref{eq:flowmap_schemes}. In other words, the inferred scheme has the same 1-step strong order as the explicit scheme it parametrizes. 
\end{theorem}
\begin{proof}
Write the flow map of the explicit scheme in parametric form: $F^\deltat(\mX_{t_n}, \Delta \mB_{t_n} ) = F^\deltat(c^*,\mX_{t_n}, \Delta \mB_{t_n} )$ as in \eqref{Fdelta}. Then, since the estimator $\widehat{c^{\delta,N,M}}$ in  \eqref{est_c} is the minimizer of the likelihood, we have 
\begin{align*}
\widehat{(\sigma^{N,M})^2}  & =  \frac{2}{d\deltat^2} \frac{1}{MN} \sum_{m=1}^M\sum_{n=0}^{N-1}  |\mX_{t_{n+1}}^{(m)} -\mX_{t_{n}}^{(m)} - \deltat  F^\deltat (\widehat{c^{\deltat,N,M} },\mX_{t_n}^{(m)},\Delta \mB_{t_n}^{(m)})|^2  \\
 & \leq  \frac{2}{d\deltat^2} \frac{1}{MN} \sum_{m=1}^M\sum_{n=0}^{N-1}  |\mX_{t_{n+1}}^{(m)} -\mX_{t_{n}}^{(m)} - \deltat  F^\deltat (c^*,\mX_{t_n}^{(m)},\Delta \mB_{t_n}^{(m)})|^2. 
\end{align*}
Since the process $(\mX_{t_n},\Delta \mB_{t_n})_n$ is stationary,  we have 
\begin{align} \label{Eresidual}
 \E \widehat{(\sigma^{N,M})^2}  & =  \frac{2}{d\deltat^2} \E |\mX_{t_{n+1}} -\mX_{t_{n}} - \deltat  F^\deltat (\widehat{c^{\deltat,N,M} },\mX_{t_i},\Delta \mB_{t_n})|^2   \\
 & \leq \frac{2}{d} \E |\frac{\mX_{t_{n+1}} -\mX_{t_{n}} }{\deltat}-  F^\deltat (c^*, \mX_{t_i},\Delta \mB_{t_n})|^2.  \notag
\end{align} 
Recall that $F^\deltat(\mX_{t_n}, \Delta \mB_{t_n} ) = F^\deltat(c^*,\mX_{t_n}, \Delta \mB_{t_n} )$ 
we have \eqref{eq:order_res}. 
\end{proof}

\begin{remark}[Order of residual for IS-RK4 and IS-EM]\label{rmk:res}
 Recall that either Euler-Maruyama scheme or the HRK4 schemes have 
$ \E \big[|\frac{\mX_{t_{n+1}} -\mX_{t_{n}}}{\deltat} -  F^\deltat (c^*, \mX_{t_i},\Delta \mB_{t_n})|^2 \big]= O(\deltat)$, which follows from Ito formula. Thus, for the inferred schemes IS-EM and IS-RK4 in \eqref{eq:flowmap_schemes}, we have $\E [\widehat{(\sigma^{N,M})^2}] = O(\deltat) $.   
Furthermore, by the Law of Large Numbers, we have $\widehat{(\sigma^{N,M})^2} \to  \E \widehat{(\sigma^{N,M})^2}  $ a.s.~when either $M\to \infty$ or $N\to \infty$. Thus, the estimator $\widehat{\sigma^{N,M}}= O(\deltat^{1/2})$ a.s.~for large $N$ or $M$. However, IS-SSBE's residual is not controlled by SSBE, because SSBE is not in the parametric family of IS-SSBE and the neglected term $R$ in \eqref{eq:res_SSBE} may prevent the residual from decaying {\rm (see Figure  \ref{fig:conv_GradCouple}(b))}.    
 \end{remark}

 \begin{remark}\label{rmk:N_vs_M} The framework of inference-based scheme applies also for non-ergodic systems to achieve reduced-in-time models. The convergence of the parameters in Theorem~{\rm\ref{thm_convEst}} and Theorem~{\rm\ref{thm:order_res}} remain true when the sample size $M$ goes to infinity. Furthermore, one may accelerate the simulation of slowly converging ergodic systems by training the inference-based schemes iteratively in time. In this study, we focus on non-globally Lipschitz ergodic systems  to highlight the ability of the inferred-scheme in producing long-term statistics.   
 \end{remark}

\begin{remark}[Optimal reduction in time]
We emphasize that our goal is to infer an explicit scheme with a relative large time-step for efficient simulation of non-globally Lipschitz ergodic systems. Theorem {\rm \ref{thm:order_res}} shows that the error in 1-step approximation (of the flow map $\mcF^\deltat$ in \eqref{eq:flowmap}) decays at the time-step decreases. But a smaller residual due to a smaller time-step does not necessarily imply a better performance for the inferred scheme, because large error may be accumulated in the invariant measure even when the time-step is small (e.g., EM may have a wrong invariant measure). To improve the inferred scheme, we seek for an optimal time-step that balances the 1-step approximation error and the accumulation into the invariant measure. In our numerical examples in the next section, the inferred scheme performs (in the sense of reproducing the invariant measure and temporal correlation) the best when the time-step is moderately large. This is similar to the parameter estimation for homogenization of multiscale process {\rm \cite{PS07}}, where the sub-sampling rate must be between the two characteristic time scales of the SDE.  One may view the inference as an averaging process through the sampling of the invariant measure and connect the error in invariant measure with sampling error, then one can find an optimal time step through a trade-off between the approximation error (i.e., numerical error) and the sampling error. We leave this as future work.  
\end{remark}

 \section{Examples}
  \label{sec:example1}

 In this section, we test three benchmark examples for each inference-based scheme proposed in \eqref{eq:IS_Euler} using two different parametric settings: $c_0$ excluded vs $c_0$ included, {so as to distinct the contribution of the linear term parametrized by $c_0$.} 
 Three non-globally Lipschitz examples are: a 1D equation with double-well potential; a 2D gradient system; and a 3D stochastic Lorenz system with degenerate noises. 
  
 In each of the examples, we generate data for inference by the Split Step Backward Euler (SSBE) scheme with very fine scale time step $\Delta t$. We infer schemes for different time step-sizes $\deltat= \Gap\times \Delta t$ with 10 options for the time gap: $\Gap \in \{ 1, 2, 4, 10, 20, 40, 80, 120, 160,200\}$, which will be used to select optimal time gap and demonstrate the convergence order of the residual in Theorem \ref{thm:order_res}. The computations of inference schemes include 5 different options
(1) IS-EM with $c_0$ excluded;
(2) IS-RK4 with $c_0$ excluded;
(3) IS-RK4 with $c_0$ included;
(4) IS-SSBE with $c_0$ excluded;
and (5) IS-SSBE with $c_0$ included. 
  
 We assess the performance of these schemes by the accuracy of the reproduced invariant density (PDF) and the auto-correlations function (ACF), which are empirically computed from a long trajectory. The accuracy of the PDF are measured by the total variation distance (TVD) from the reference PDF from data. 
 
 Once we identify the best performing scheme for each example, we fix the inference settings and present the convergence of the estimators and the residuals with respect to the time $\Gap$ as well as the number of trajectories $M$.
 
 \bigskip
 In summary, we find from the examples that 
 \begin{itemize}
     \item The inferred scheme has significantly stronger numerical stability than the plain schemes. The IS-RK4 and IS-SSBE exhibit better stability than IS-EM. In particular, they can tolerate time-steps that are significantly larger than the plain RK4 or SSBE. Specifically, we find the plain RK4 and SSBE (without inferred parameters) always blow up even when $\Gap=20$, whereas the inferred schemes are still stable when $\Gap$ is larger than 200, which improves the efficiency by an order of more than 10. We summarize the blow-up gap for plain verse inferred schemes for each example in the following table.
\begin{table}[htp!]
 \centering \vspace{-1mm}
    \begin{tabular}{|c|c|c|c|}
    \hline
    {}&{1D double-well}&{2D gradient system}&{3D Lorenz system}\\
    \hline
    Plain RK4 &$\Gap=20$ & $\Gap=20$ & $\Gap=10$\\
    \hline
    IS-RK4 &$\Gap>200$ & $\Gap>200$ & $\Gap>400$\\
    \hline
    Plain SSBE &$\Gap=40$ & $\Gap=40$ & $\Gap=20$\\
    \hline
    IS-SSBE &$\Gap>200$ & $\Gap>200$ & $\Gap>400$\\
    \hline
    \end{tabular}
  \vspace{-3mm}     \caption{Blow up time gap for each scheme: plain verse inferred.
    \label{tab:blow_up_gap}}
\end{table}

     \item The inferred scheme can reproduce the invariant measure accurately. Both IS-RK4 and IS-SSBE perform well when the stochastic force dominates the dynamics. But when the drift dominates the dynamics in the example of Lorenz system, IS-RK4 performs better than IS-SSBE, because it provides a better approximation to the drift than IS-SSBE. 
     
     \item The inferred scheme reproduces the invariant density the best when the time-step is medium large (with a time gap between $\Gap=80$ and $\Gap=160$), suggesting a balance between the approximation error of the flow map and the numerical error in simulating the invariant density. It is open to have an \textit{a-priori} estimate of the optimal time gap.
 \end{itemize}
 



\subsection{1D double-well potential}
First consider an 1D SDE with a double-well potential \cite{MSH02}
\begin{equation}\label{example:1D_DoubleWell}
dX_t=- V'(X_t)dt+\sqrt{2/\beta} dB_t,
\end{equation}
with $V(x)=\frac{\mu}{4}(x^2-1)^2$.
The corresponding invariant measure is $\frac{1}{Z}\exp^{-\beta V(x)}$
where $Z$ being the normalizing constant $Z:=\int_{\mathbb{R}}\exp^{-\beta V(x)}dx$. We set $\mu=2$ and $\beta=1$. 

We generate data by SSBE with a fine time-step $\Delta t=1e-3$. We first simulate a long trajectory on an interval $[0,T]$ with $X_0=1/2$ and $T=2000$ (i.e., two million time steps), which is found to be long enough to represent the invariant density (PDF). This long trajectory will also provide us the reference PDF and ACF, which are referred as the true values to be approximated. 
Then we generate $M=1000$ trajectories on the time interval $[0,1000]$  with initial conditions sampled from the long trajectory. The data for inference are the $M$ trajectories of both the Brownian motion and the process $(X_{t})$ observed at discrete times $\{t_n=n\deltat = n\Gap \times \Delta t\}$, as in \eqref{eq:data}. 

The parameters of the schemes in \eqref{eq:IS_Euler} are then estimated by Algorithm \ref{alg:main} for each  $\deltat = \Gap \times \Delta t$.  

 Figure~\ref{fig:TVDPDF_LangLocal1D}(a) shows the TVD of the five schemes with time gaps $\Gap\in \{10, 20, 40, 80, 120, 160,200\}$. 
 Note that for every scheme, the TVD first decreases and then increases, reaching the smallest TVD when $\Gap=80$. This suggests that when the gap is small, the approximation error of the flow map (recall that the data are from an implicit scheme while the inferred schemes are explicit schemes) dominates the error in the invariant measure; when the gap is large, the numerical error of the inferred schemes dominates the TVD. A balance between the two errors is reached at the medium large time-step. 
 
 We first select the scheme that reproduces the invariant density with the smallest TVD.  Overall, the IS-RK4 schemes perform the best and the inclusion of $c_0$ brings in negligible improvement. Thus, we select IS-RK4 without $c_0$ to demonstrate further results.  
 
\begin{figure}[htp!]
 \centering \vspace{-1mm}
    \subfigure[TVD]{\includegraphics[width=0.30\textwidth]{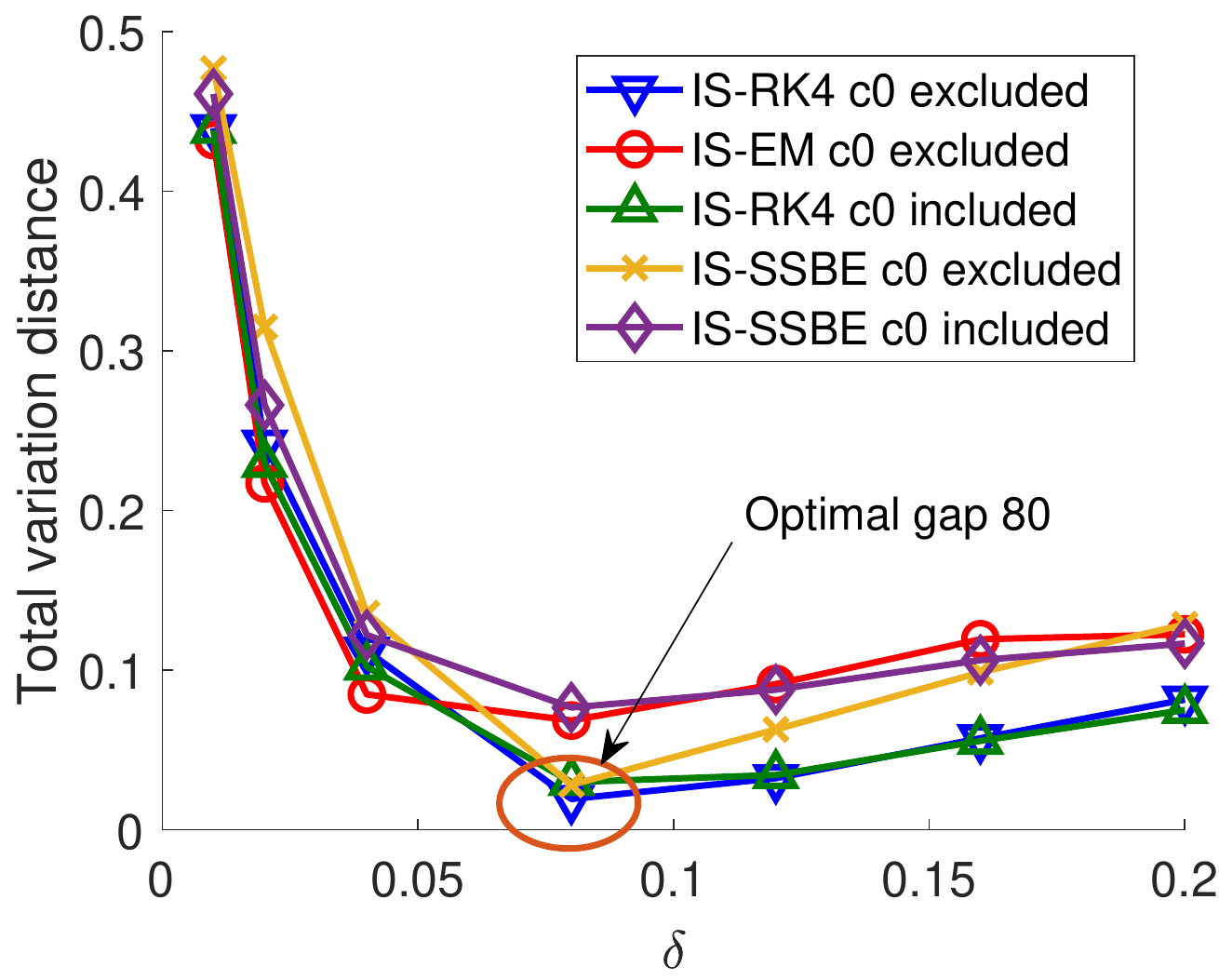}}
     \subfigure[PDF]{\includegraphics[width=0.32\textwidth]{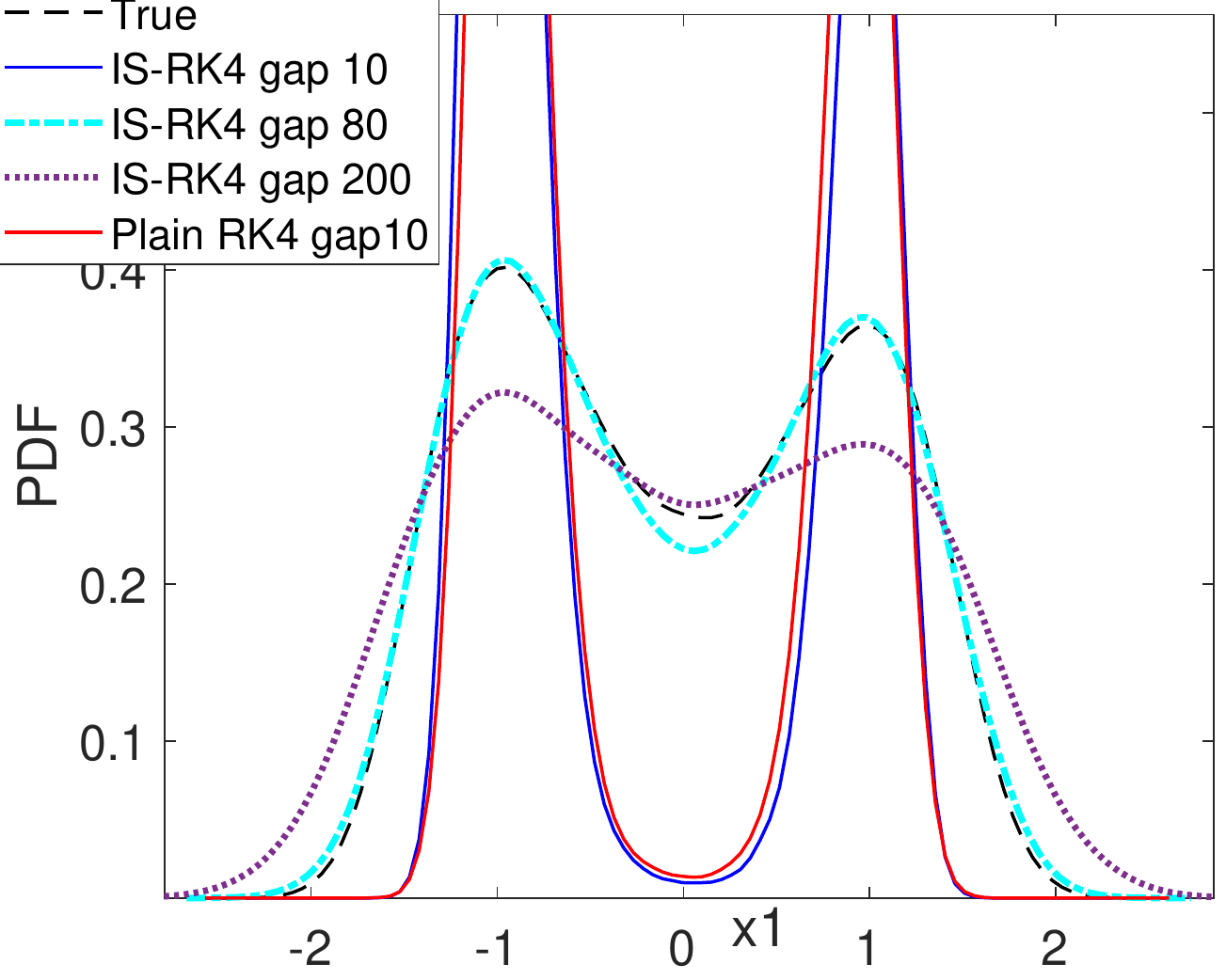}}
     \subfigure[ACF]{\includegraphics[width=0.32\textwidth]{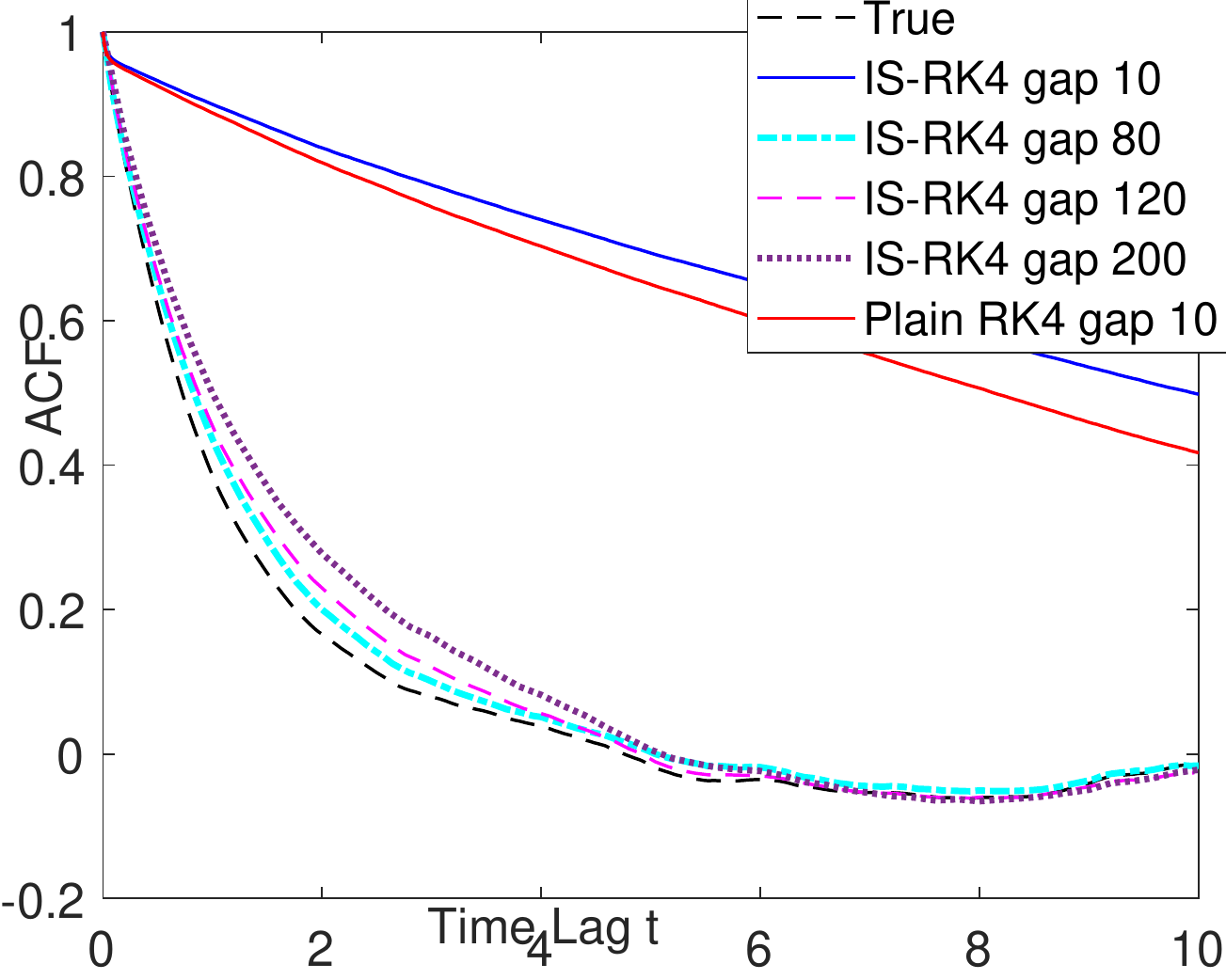}}
 \vspace{-3mm}     \caption{Large-time statistics for 1D double-well potential.  (a) TVD between the empirical invariant densities (PDF) of the inferred schemes and the reference PDF from data. (b) and (c):  PDFs and ACFs comparison between the IS-RK4 with $c_0$ excluded and the reference data.
    \label{fig:TVDPDF_LangLocal1D}}
\end{figure}

Figure~\ref{fig:TVDPDF_LangLocal1D} (b-c) show the PDFs and auto-correlation functions (ACFs) of IS-RK4 with $c_0$ excluded at three representative time gaps $\Gap\in\{10,80,200\}$, in comparison with those of the reference data and the plain RK4 with $\Gap=10$. 
When $\Gap$ is small, that is $\Gap=10$, the IS-RK4 is close to the plain RK4, and both produce PDFs and ACFs with large errors. The PDF and ACF generated by IS-RK4 with $\Gap=80$ is the best among all used gaps, fitting the true PDF and ACF almost perfectly. Furthermore, when time gap is as large as $\Gap=200$, the IS-RK4 can still produce qualitative results with the feature of PDF (that is the double-well feature), whereas the plain RK4 scheme blows up when $\Gap=20$.

\begin{figure}[htp!]
 \centering \vspace{-1mm}
         \subfigure[Convergence of $c_1$ in sample size]{
    \includegraphics[width=0.39\textwidth]{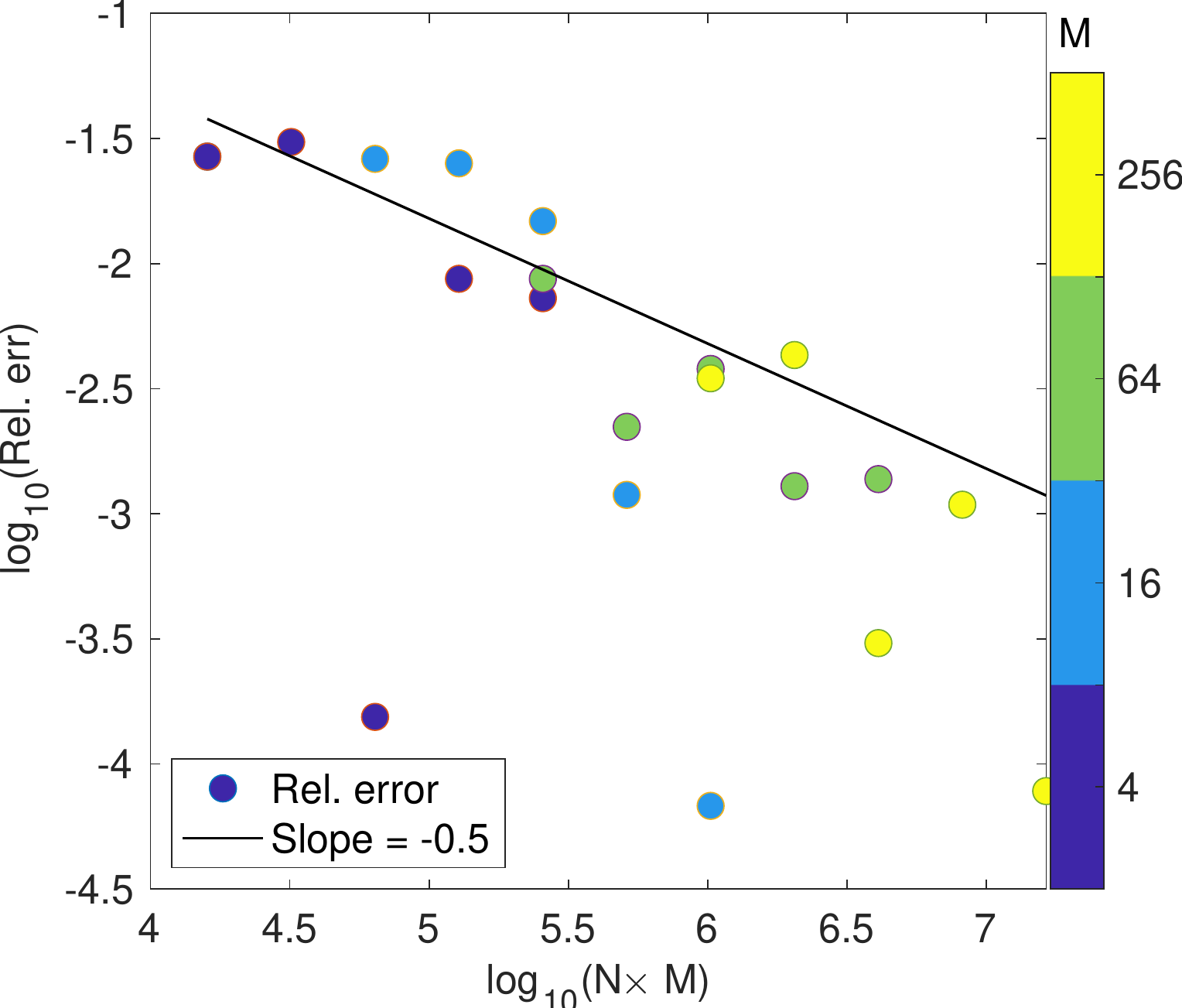}
    }
    \subfigure[Coefficients and residuals]{
    \includegraphics[width=0.58 \textwidth]{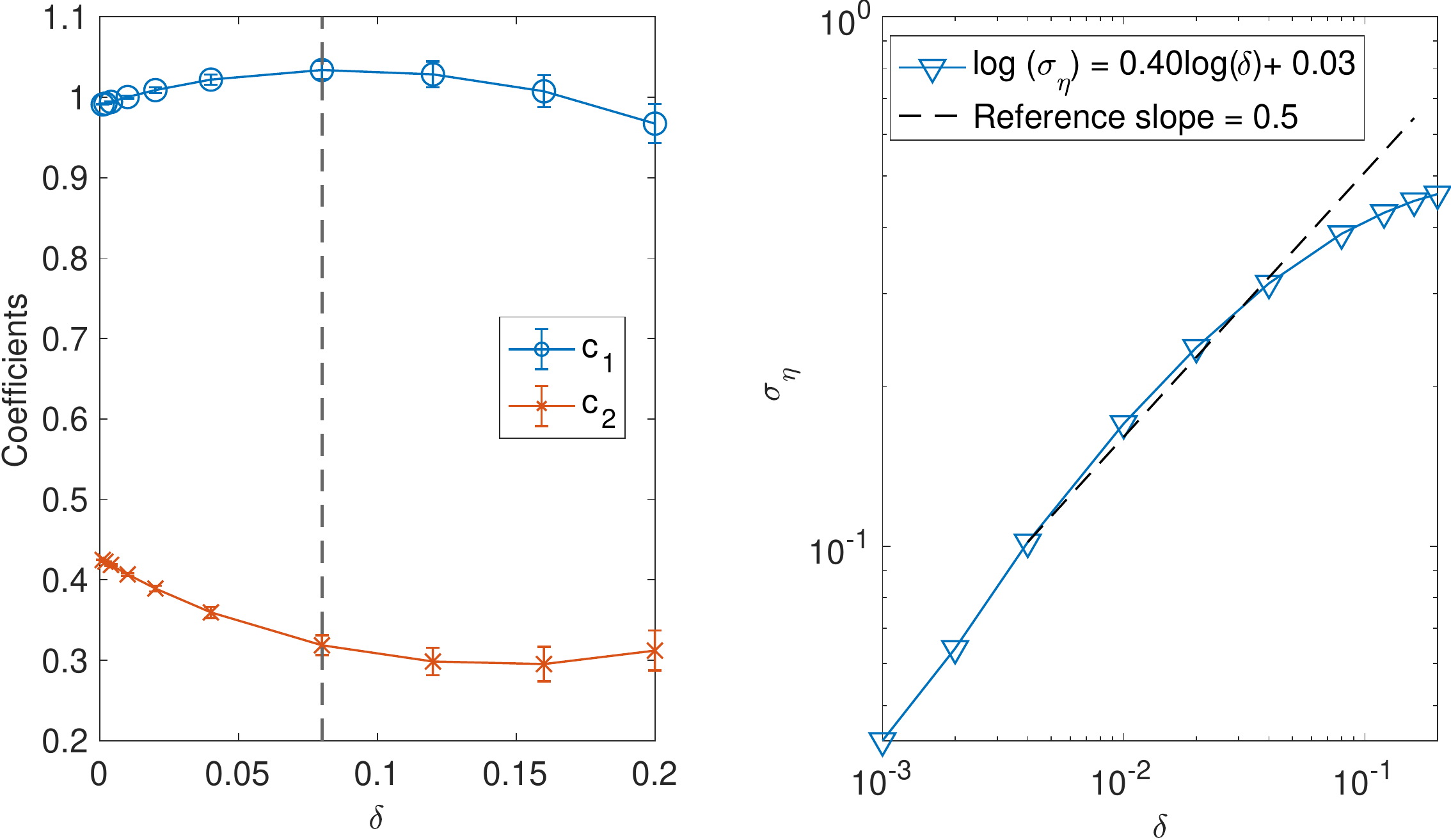}}
   
 \vspace{-3mm} \caption{1D double-well potential: Convergence of estimators in IS-RK4 with $c_0$ excluded. (a) The relative error of the estimator $\widehat{c_1^{\deltat, N,M}}$ with $\deltat = 80\times \Delta t$ converges at an order about $(MN)^{-1/2}$, matching Theorem \ref{thm_convEst}. (b) Left column: The coefficients depend on the time-step $\deltat = \Gap\times \Delta t$, with $c_1$ being almost 1 and $c_2$ being close to linear in $\deltat$ until $\deltat>0.08$. The error bars, which are  too narrow to be seen, are the standard deviations of the single-trajectory estimators from the $M$-trajectory estimator. Right column: The residual decays at an order $O(\deltat^{1/2})$, 
    matching Theorem \ref{thm:order_res}. }
    \label{fig:conv_LangLocal1D}
\end{figure}
We also test the convergence of the estimators in sample size and their dependence on the time-step, as well as the order of residual, aiming to confirm the theory in Section \ref{sec:conv}. Figure \ref{fig:conv_LangLocal1D}(a) shows that the relative error of $\widehat{c_1^{\deltat, N,M}}$ converges at a rate about $(MN)^{-1/2}$  as the sample size $N$ or $M$ increases. Here we take the estimator from the largest sample size as the projection coefficient, and compute the relative error to it. Note that the estimator of $c_1$ is close to 1. Thus, the estimator $\widehat{c_1^{\deltat, N,M}}$ converges at a rate  about $(MN)^{-1/2}$, matching Theorem \ref{thm_convEst}. The convergence of the estimator of $c_2$ has similar convergence. 

Figure \ref{fig:conv_LangLocal1D}(b) shows the dependence of the estimators on the time-step $\deltat = \Gap\times \Delta t$. The coefficient $c_1$ is almost 1, while $c_2$ is close to linear in $\deltat$. Furthermore, it also shows that the estimators from each single trajectory are close to the M-trajectory estimator, with small standard deviations represented by error bars that are too narrow to be seen. The residual decays at an order about $0.49$ with respect to $\deltat$, closely matching the rate in Theorem~\ref{thm:order_res}.



\subsection{A 2D gradient system}
    
We now consider a 2D dissipative gradient system  \cite{MSH02}
\begin{equation}\label{example:GradCouple}
d\Xb_t=-\nabla V(\Xb_t)dt+\sqrt{2/\beta} d\mathbf{B}_t,
\end{equation}
with $V(\Xb)=V(x_1, x_2)=\exp\left(\frac{\mu_1}{2}x_1^2+\frac{\mu_2}{2} x_2^2\right)$.
The corresponding invariant measure is $\frac{1}{Z}exp^{-\beta V(x_1,x_2)}$
where $Z$ being the normalizing constant $Z:=\int_{\mathbb{R}^2}exp^{-\beta V(x_1,x_2)}dx_1dx_2$. We set $\mu_1=0.1$, $\mu_2=1$ and $\beta=2$. Because $\mu_2=10 \mu_1$, so $x_1$ is a slowly evolving variable compared to $x_2$ and the resulting dynamics displays a multi-scale feature. Consequently, we estimate parameters of the inferred schemes entry-wisely and we focus on the marginal invariant density of $x_1$. 

We generate data by the SSBE scheme with $\Delta t=2e-3$ and time interval $[0, 2000] $ with total time steps $tN=1e6$. The rest setting and procedure are the same as the 1D double-well potential case.

 Figure~\ref{fig:TVDPDF_GradCouple}(a) shows that IS-RK4 and IS-SSBE schemes have comparable TVD, and they reach the minimal TVD when $\Gap = 120$, where IS-EM blows put.  They produce similar PDFs and ACFs, so we only present those of IS-SSBE with $c_0$ excluded.  Figure~\ref{fig:TVDPDF_GradCouple}(b-c) show the PDFs and ACFs at representative time gaps $\Gap\in \{10, 80, 120, 200\}$. The findings are similar to those for the 1D double-well potential: (i) the performance of IS-SSBE first improves and then deteriorates as $\Gap$ increases; (ii) IS-SSBE can tolerate significantly larger time-step than the plain SSBE, {where the plain SSBE blows up due to the Newton-Raphson method used as the implicit solver,} {which can only tolerate a small time-step limited by the inversion (similar to \eqref{phi1_BEs}) in the Newton-Raphson method in the implicit solver. 
 }


\begin{figure}[htp!]
 \centering \vspace{-1mm}
    \subfigure[TVD]{\includegraphics[width=0.30\textwidth,height= 4cm]{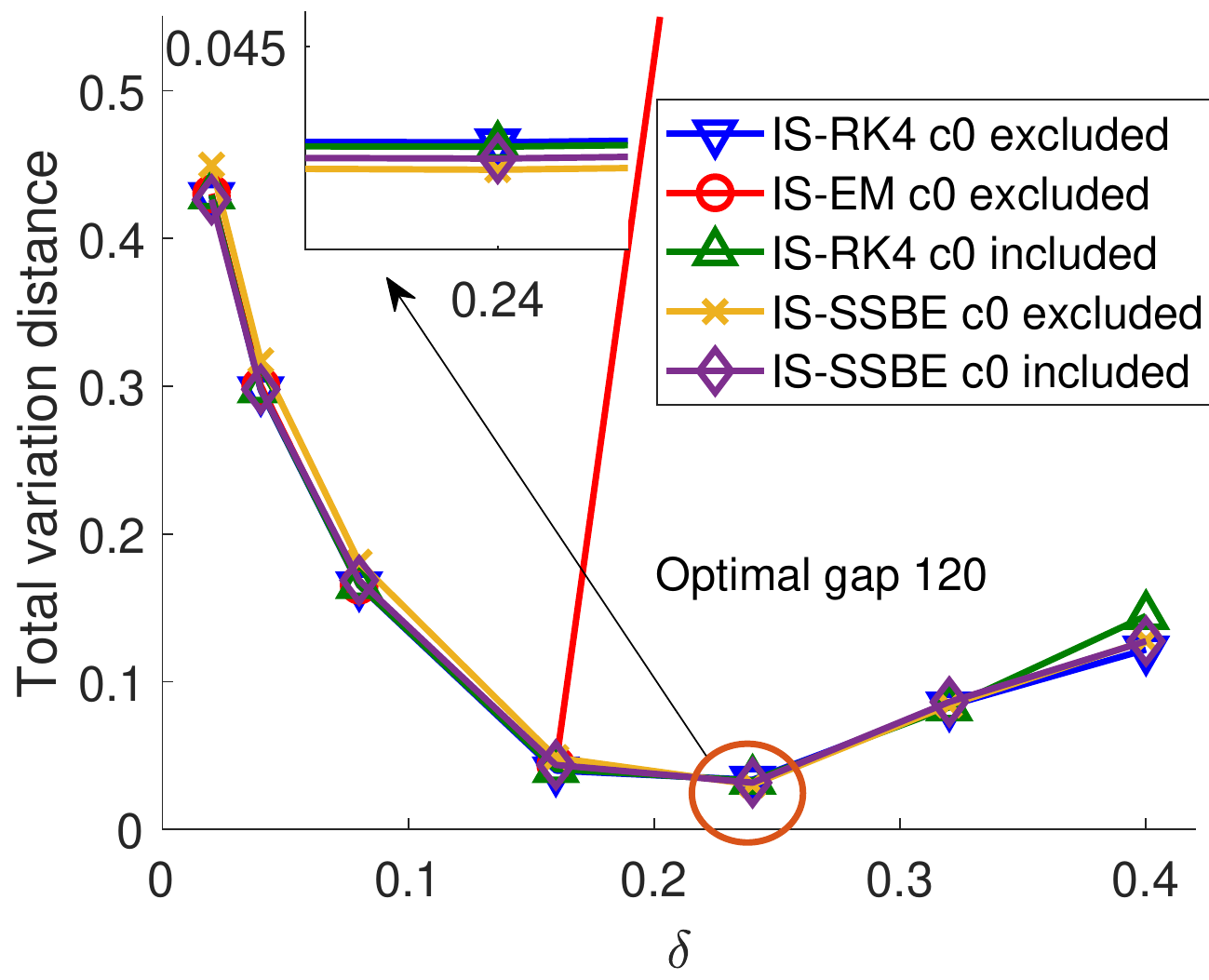}}
     \subfigure[PDF]{\includegraphics[width=0.32\textwidth,height= 4cm]{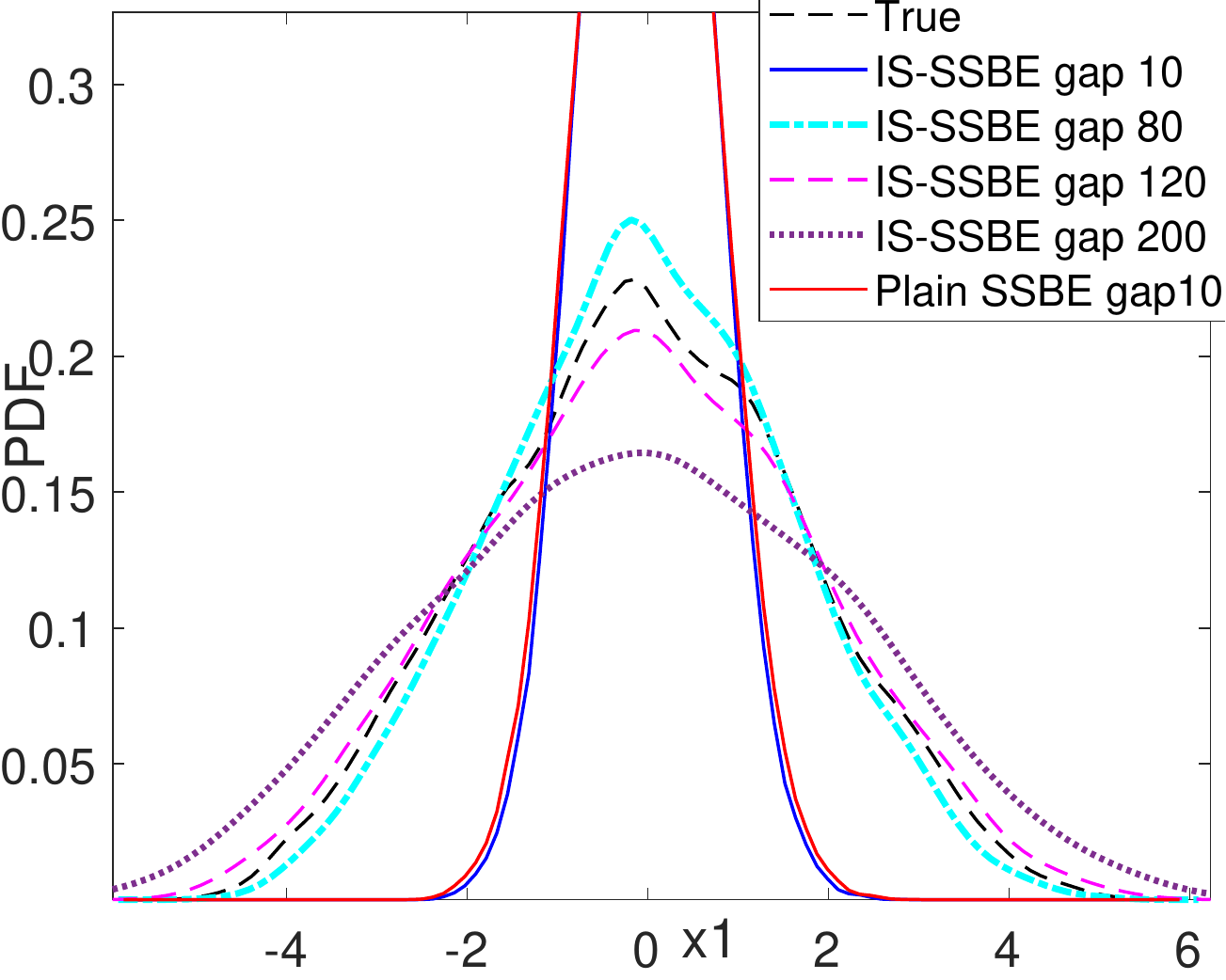}}
     \subfigure[ACF]{\includegraphics[width=0.32\textwidth,height= 4cm]{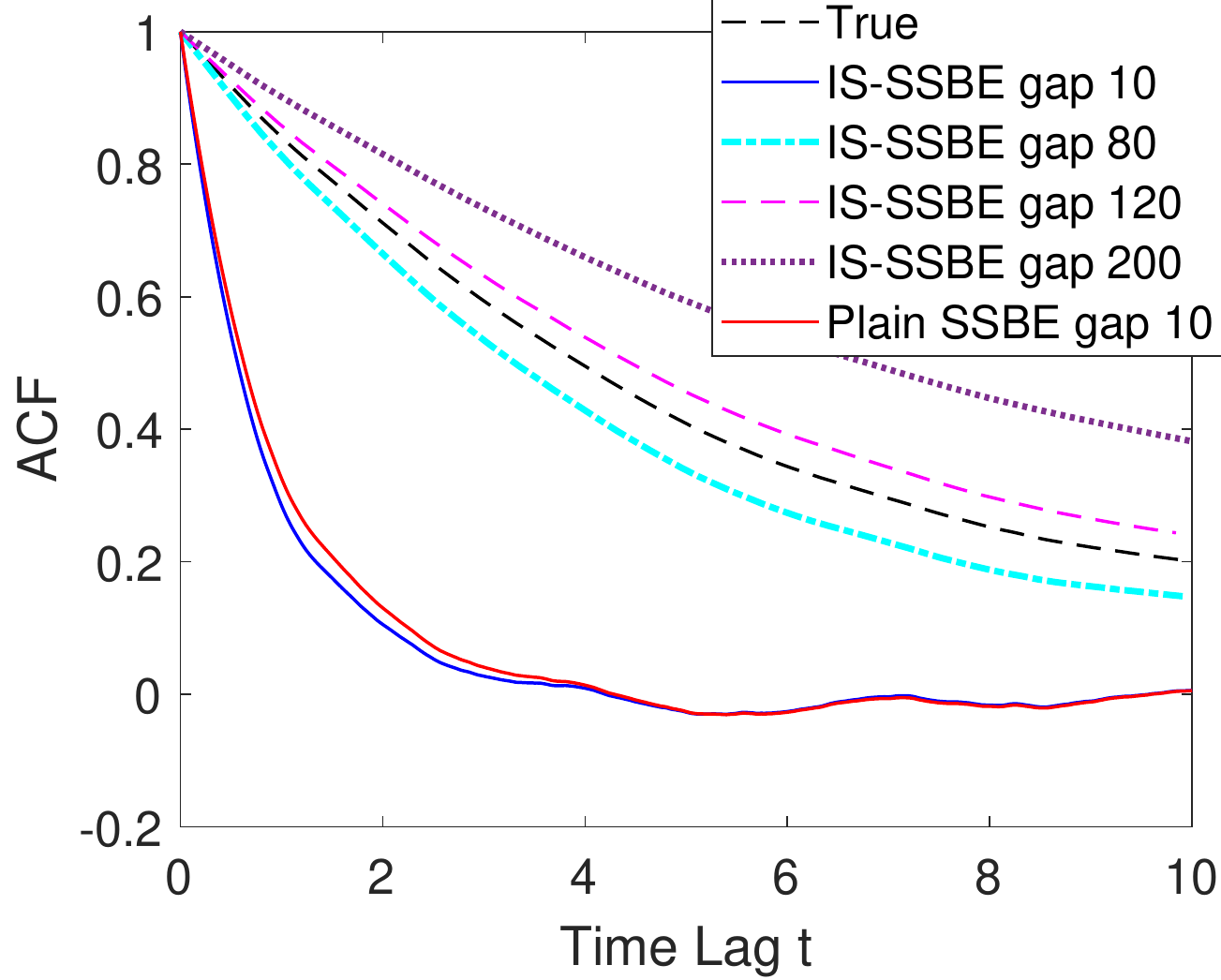}}
 \vspace{-3mm}     \caption{Large-time statistics for the 2D gradient system. (a) TVD between the $x_1$ marginal invariant densities (PDF) of the inferred schemes and the reference PDF from data. (b) and (c):  PDFs and ACFs comparison between IS-SSBE with $c_0$ excluded and the reference data.  }
    \label{fig:TVDPDF_GradCouple}
\end{figure}


The convergence of the estimators in sample size is also roughly of order $(MN)^{-1/2}$, as shown in Figure \ref{fig:conv_GradCouple}(a). Figure \ref{fig:conv_GradCouple}(b) shows that the estimators of $c_1$ and $c_2$ depend almost linearly on $\deltat$. Also, $c_1$'s single-trajectory estimators have negligible standard deviations from the $M$-trajectory estimator, while $c_2$'s estimators have a persistent noticeable standard deviation. This suggests that IS-SSBE has large uncertainties in the stochastic force term (recall that $c_1$ and $c_2$ being the coefficients of the scaled drift and the stochastic force, see \eqref{eq:IS_Euler}). In the right column, the residual of IS-SSBE remains little changed when $\deltat$ decreases, far from a decay rate $0.5$. This does not violate Theorem \ref{thm:order_res}, which is for parametrizations of explicit schemes. Instead, this highlights that the IS-SSBE is not a parametrization of the SSBE implicit scheme, and it has a flow map $\widetilde F_{SSBE}^{\deltat}$ with distance to the true flow map $\E \big[|\frac{\mX_{t_{n+1}} -\mX_{t_{n}}}{\deltat} -  \widetilde F_{SSBE}^{\deltat} (c, \mX_{t_i},\Delta \mB_{t_n})|^2 \big]$ depending little on $\deltat$. Such a feature may be helpful for further efforts on improving the parametric form.


\begin{figure}[htp!]
 \centering \vspace{-1mm}
    \subfigure[Convergence of $c_1$ in sample size]{
    \includegraphics[width=0.38\textwidth]{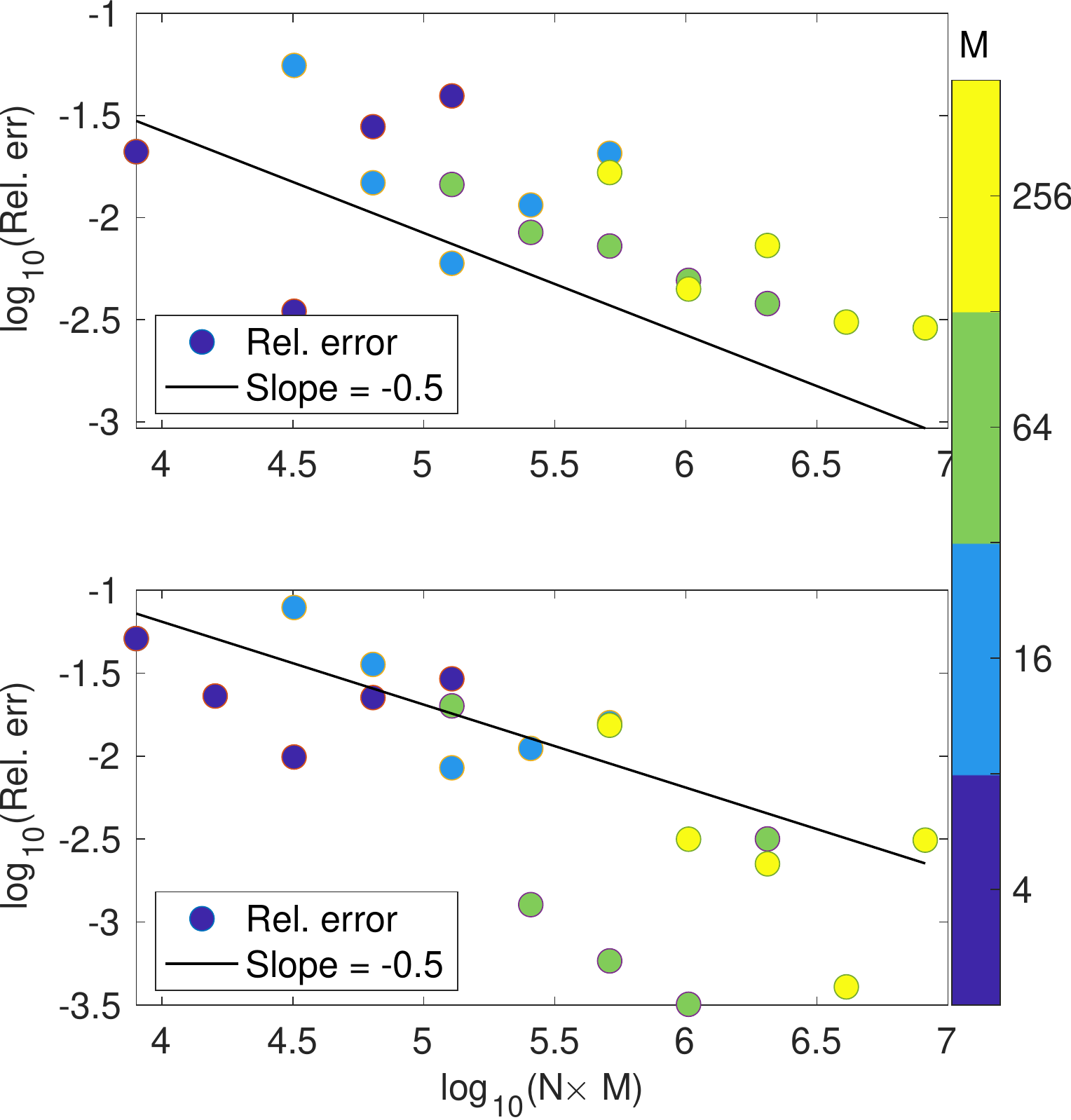}}
    \subfigure[Coefficients and residuals]{
    \includegraphics[width=0.58 \textwidth]{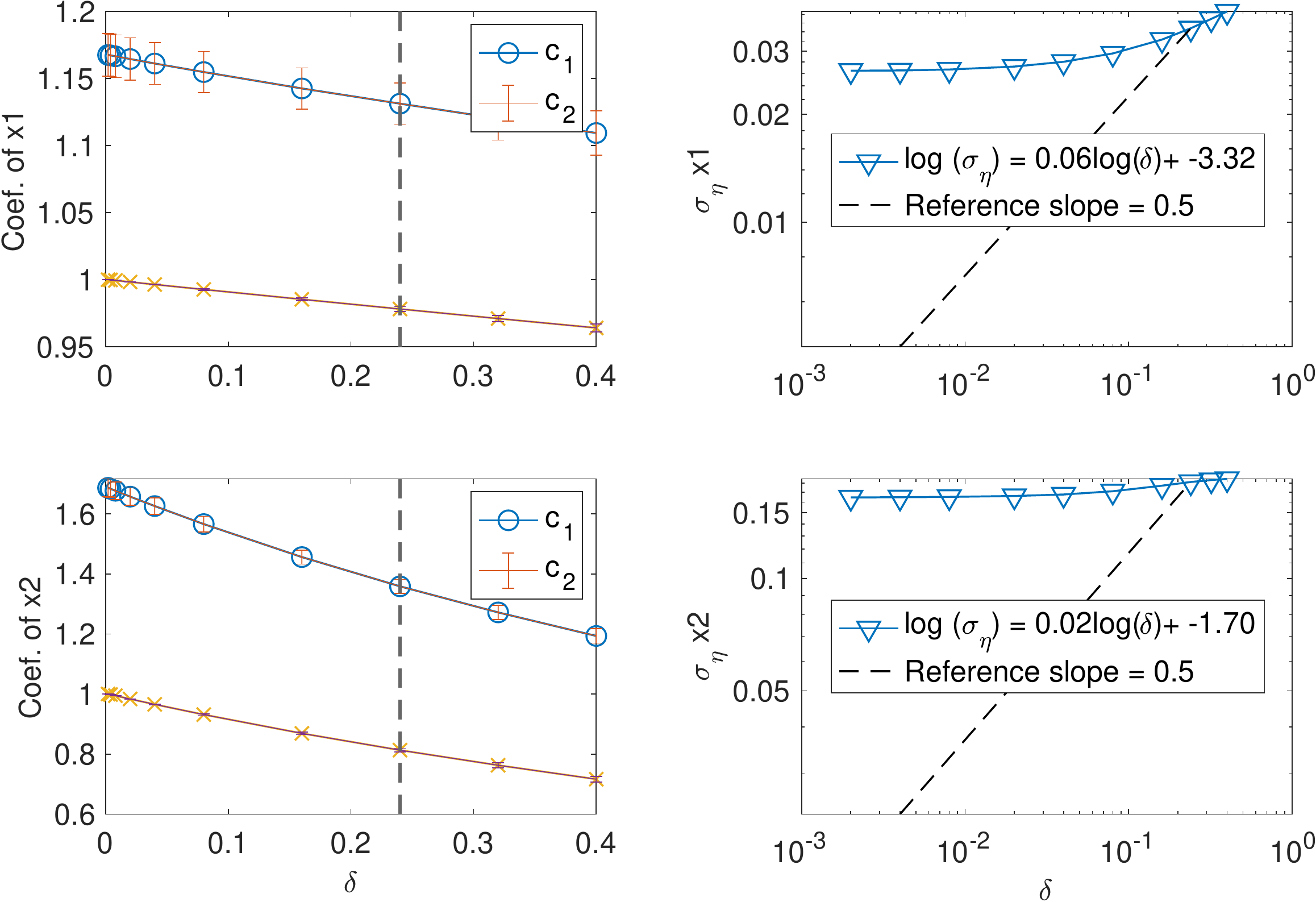}}
    
 \vspace{-3mm}     \caption{2D gradient system: Convergence of estimators in IS-SSBE with $c_0$ excluded. (a) The relative error of the estimator   $\widehat{c_1^{\deltat, N,M}}$ with $\deltat = 120 \Delta t$ converges at an order about $(MN)^{-1/2}$, matching Theorem \ref{thm_convEst}. (b) Left column: The estimators of $c_1, c_2$ are almost linear in $\deltat$. Right column: The residual changes little as $\deltat$ decreases, due to that IS-SSBE is not a parametrization of an explicit scheme (thus, Theorem~\ref{thm:order_res} does not apply).     
    \label{fig:conv_GradCouple} }
\end{figure}

Figure~\ref{fig:conv_GradCouple_RK4} shows the convergence of the estimator for IS-RK4. Similar to the 1D case, we observe a convergence rate $(MN)^{-1/2}$ in Figure~\ref{fig:conv_GradCouple_RK4}(a). Also,  in Figure~\ref{fig:conv_GradCouple_RK4}(b), we observe almost $\deltat$ independent estimators and the expected decay rate $O(\deltat^{1/2})$ proved in Theorem \ref{thm:order_res}. 


\begin{figure}[htp!]
 \centering \vspace{-1mm}
     \subfigure[Convergence of $c_1$ in sample size]{
    \includegraphics[width=0.38\textwidth]{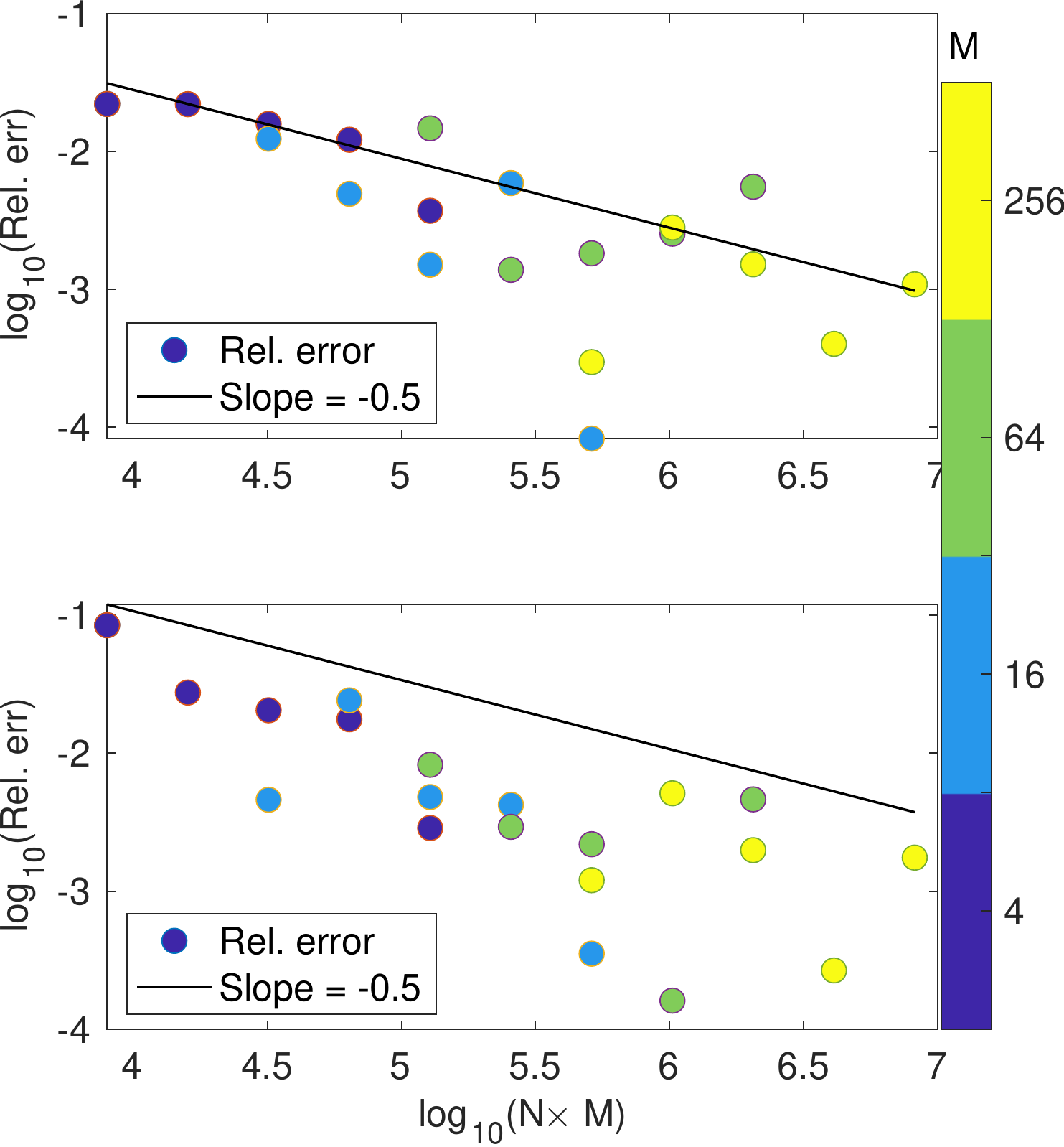}}
      \subfigure[Coefficients and residuals]{
    \includegraphics[width=0.58 \textwidth]{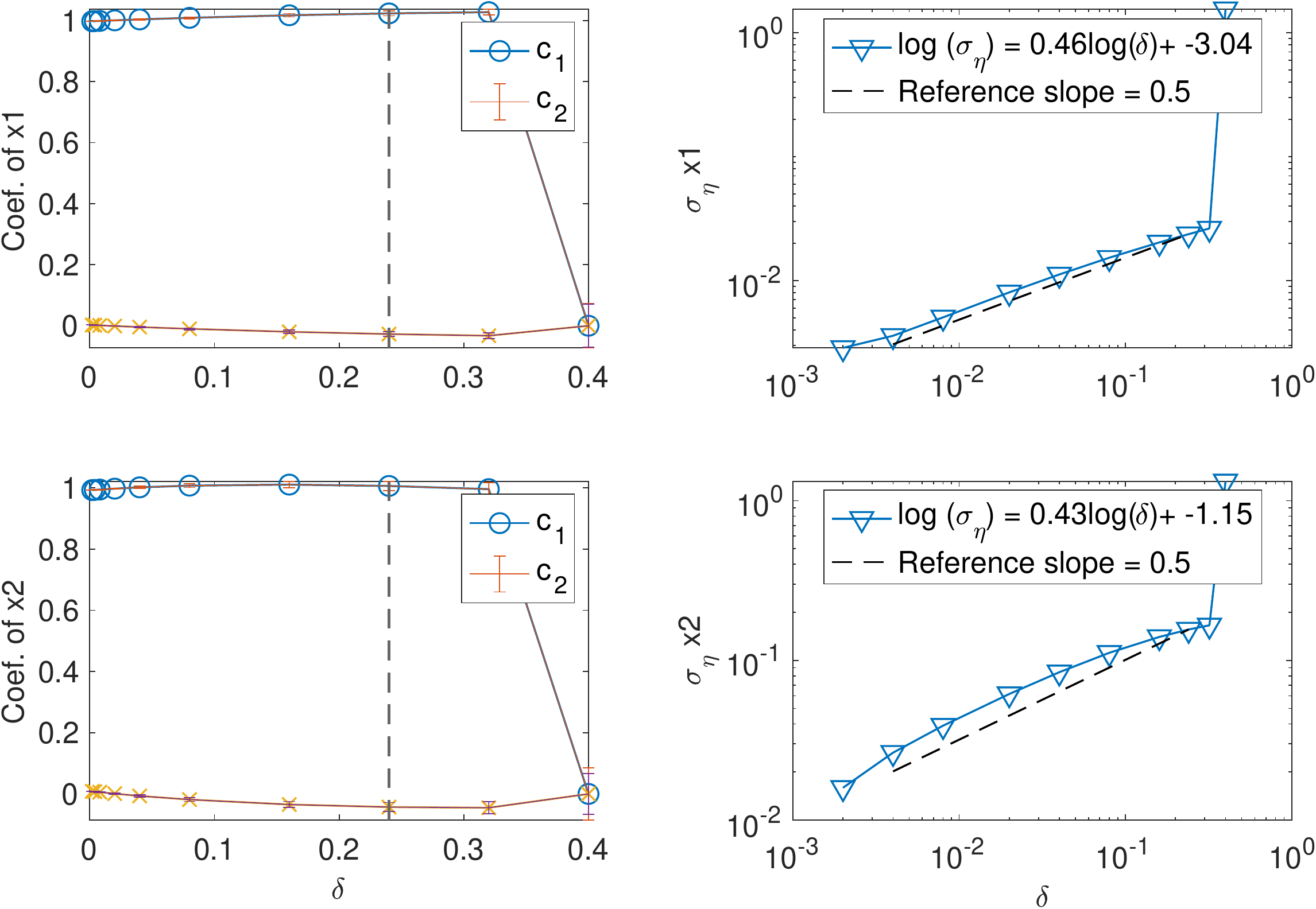}}
 \vspace{-3mm}     \caption{2D gradient system: Convergence of estimators in IS-RK4 with $c_0$ excluded. (a) The relative error of the estimator   $\widehat{c_1^{\deltat, N,M}}$ with $\deltat = 120 \Delta t$ converges at an order about $(MN)^{-1/2}$, matching Theorem \ref{thm_convEst}. (b) Left column: The estimators of $c_1, c_2$ are constant for all $\deltat$. Right column: The residual decays at an order $O(\deltat^{1/2})$, 
    matching Theorem \ref{thm:order_res}. 
    \label{fig:conv_GradCouple_RK4} }
\end{figure}

 \subsection{Stochastic Lorenz system with degenerate noise}
Consider next the 3D stochastic Lorenz system with degenerate noise \cite{MSH02}
\begin{equation}\label{example:Lorenz}
\begin{aligned}
dx_1 & =\sigma(x_2-x_1)dt+\sqrt{2/\beta}dB_1,\\ 
dx_2 & =\big( x_1(\gamma-x_3)-x_2\big)dt+\sqrt{2/\beta}dB_2,\\ 
dx_3& =(x_1x_2-bx_3)dt.
\end{aligned}
\end{equation}
We set $\sigma=10$, $\gamma=28$, $b=8/3$ and $\beta=1$. 
This stochastic chaotic system is exponentially ergodic with a regular invariant measure because it is dissipative and hypoelliptic. 


As before, we generate data by SSBE with $\Delta t=5e-4$ and a reference long trajectory with $tN=6e6$ time steps (or equivalently, on the time interval $[0,3000]$). We consider time gaps $\Gap\in \{20, 40, 80, 160, 240, 320,400\}$, so the maximal time-step is still $0.2$.  

\begin{figure}[htp!]
 \centering \vspace{-1mm}
    \subfigure[TVD]{\includegraphics[width=0.30\textwidth,height= 4cm]{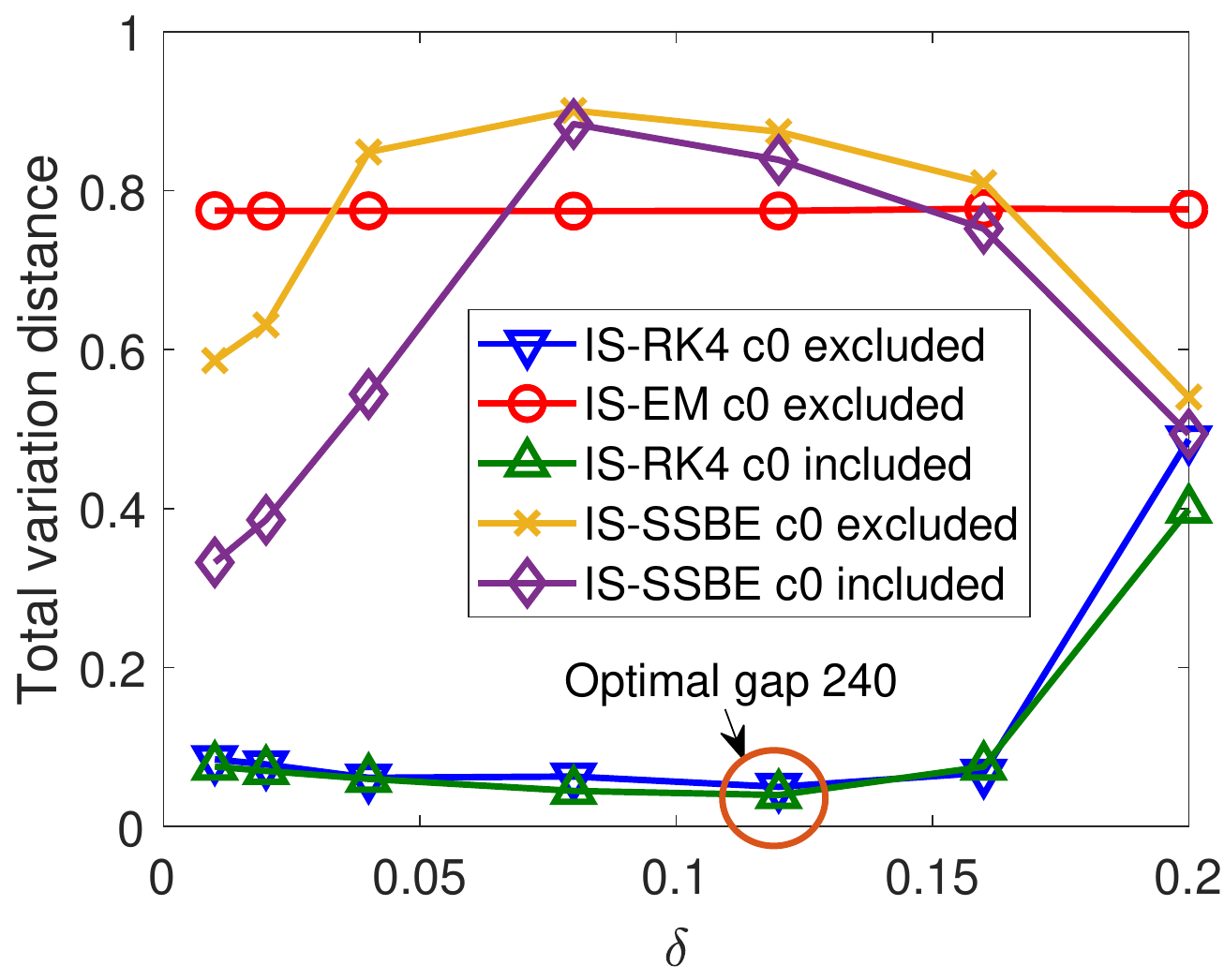}}
     \subfigure[PDF]{\includegraphics[width=0.32\textwidth,height= 4cm]{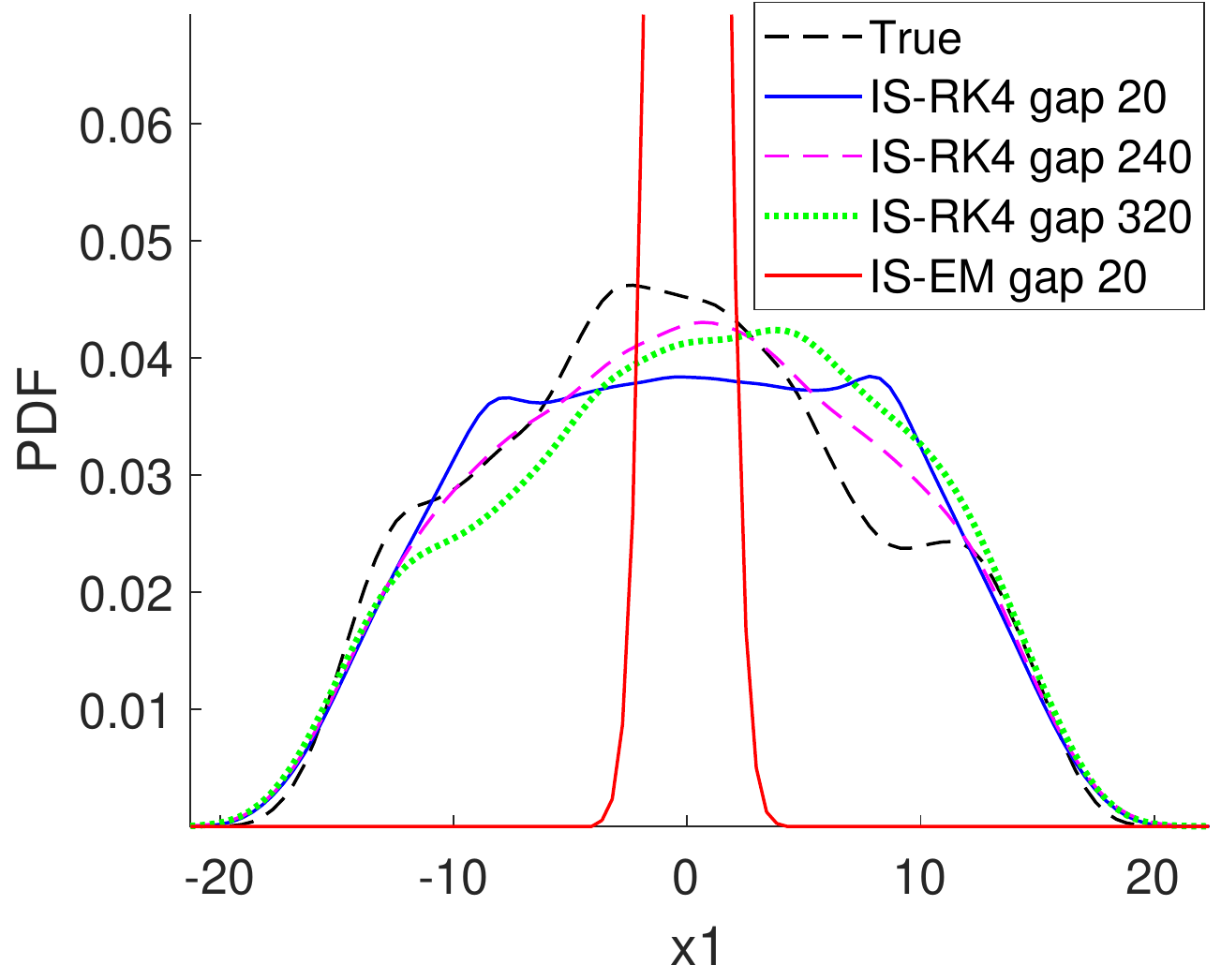}}
     \subfigure[ACF]{\includegraphics[width=0.32\textwidth,height= 4cm]{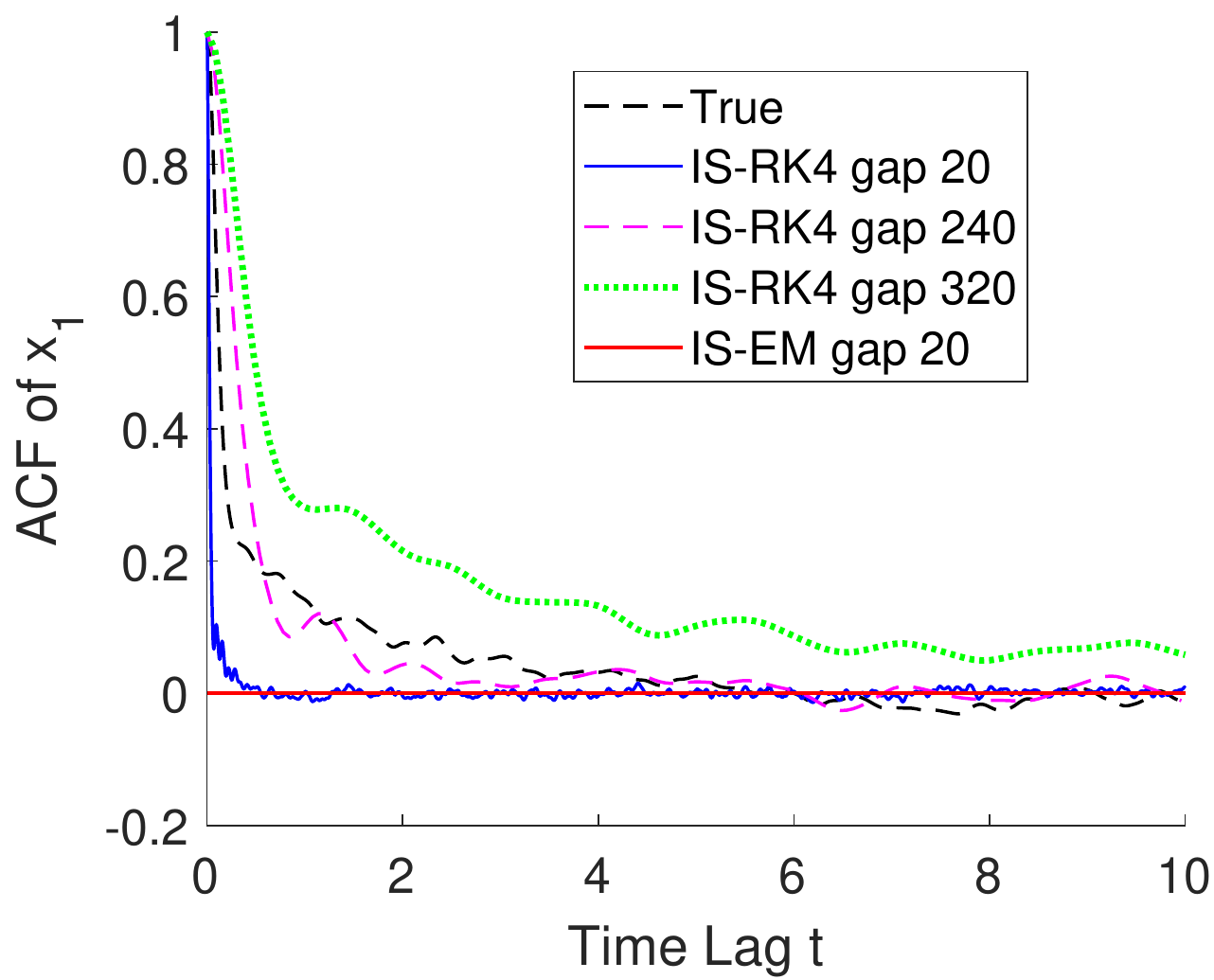}}
 \vspace{-3mm}     \caption{Large-time statistics of $x_1$ for the stochastic Lorenz system. (a) TVD between the $x_1$ marginal invariant densities (PDF) of the inferred schemes and the reference PDF from data. (b) and (c): PDFs and ACFs  comparison between IS-RK4 with $c_0$ included and the reference data.}
    \label{fig:TVDPDF_Lorenz}
\end{figure}

Figure~\ref{fig:TVDPDF_Lorenz}(a) shows the TVD of the inferred schemes. This time, the IS-RK4 scheme performs significantly better than IS-SSBE schemes, 
with relatively small TVD for most time gaps. This is due to the high-order approximation of RK4 to the drift, particularly when the drift dominates the dynamics (note that the state variable $x_1$ is at a scale of magnitude larger than the degenerate noise). The IS-RK4 with $c_0$ included performs the best and we select it for further demonstration of results. 

Figure~\ref{fig:TVDPDF_Lorenz}(b-c) show the PDFs and ACFs at representative time gaps $\Gap\in \{20, 240, 320\}$. Since the plain RK4 blows up at $\Gap=10$, so we display the results from IS-EM instead. The findings are similar to those for the 1D double-well potential: (i) the performance of IS-RK4 first improves and then deteriorates as $\Gap$ increases; (ii) IS-RK4 can tolerate significantly larger time-step than the plain RK4. 


\begin{figure}[htp!]
 \centering \vspace{-1mm}
    \subfigure[PDF]{\includegraphics[width=0.32\textwidth,height= 4.3cm]{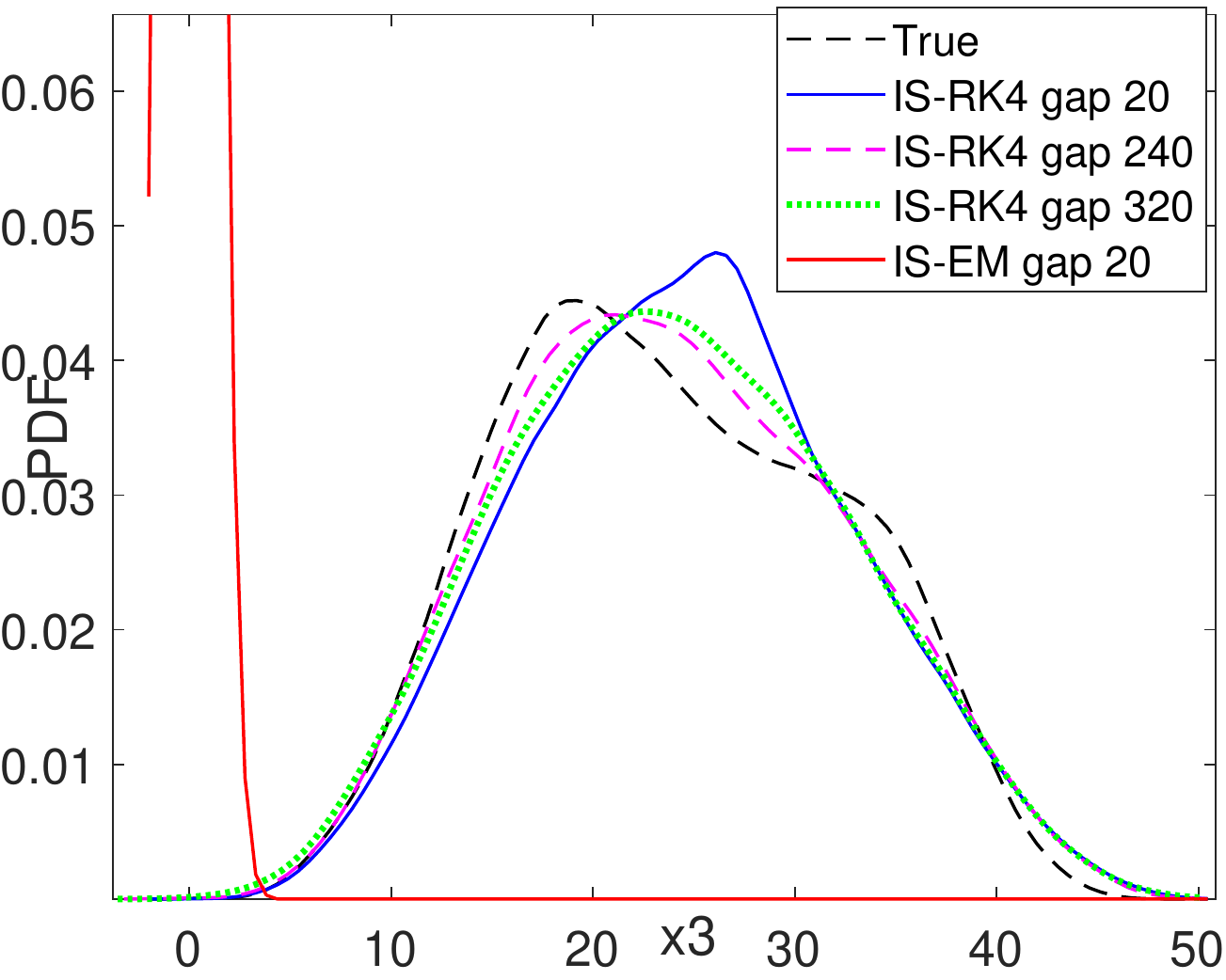}}
     \subfigure[ACF]{\includegraphics[width=0.32\textwidth,height= 4.3cm]{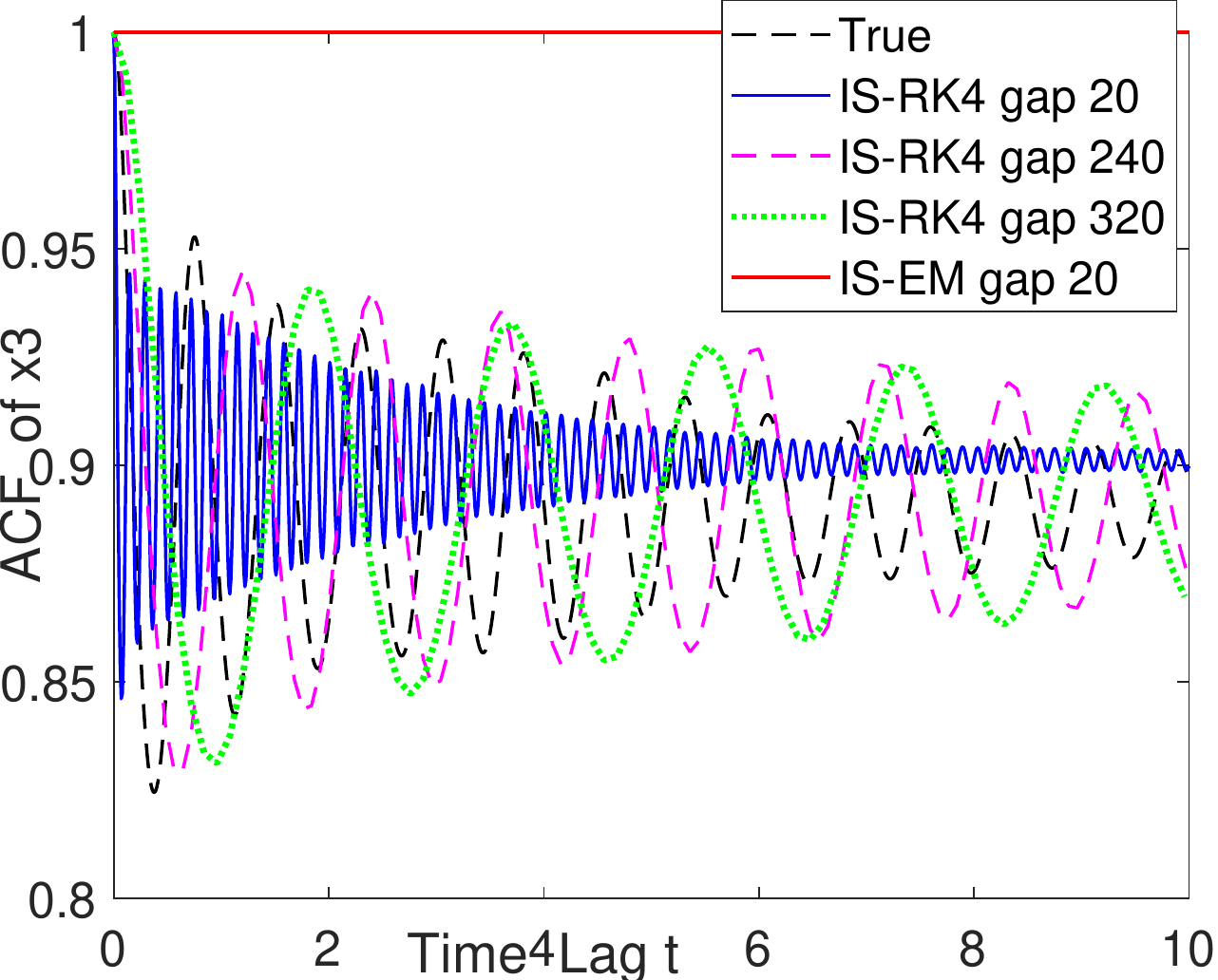}}
 \vspace{-3mm}     \caption{ACF and PDF of $x_3$ in the stochastic Lorenz system. Similar to the other examples, IS-RK4 (with $c_0$ included) reproduces the PDF and the ACF the best when the time-step is medium large, while plain RK4 and IS-EM blow up even when $\Gap=20$.}
    \label{fig:TVDPDF_Lorenz_x3}
\end{figure}

Moreover, we also plot the PDF and ACF of $x_3$ in Figure~\ref{fig:TVDPDF_Lorenz_x3}. The dynamics of $x_3$ is the most challenging because there is no diffusive stochastic force acting on it and its ACF is highly oscillatory. As usual, the IS-RK4 can reproduce the PDF and ACF well, whereas the plain RK4 and IS-EM blow up even when the time-step is small. In particular, the IS-RK4 produce the periodic and decay feature of $x_3$'s ACF when the time-step is medium large, that is $\Gap=240$. We expect the best performance to be achieved at a gap between 120 to 240, and we postpone the study on the optimal time gap and other improvements in future work. 

The IS-RK4 has convergence results  mostly as expected. Figure~\ref{fig:conv_res_Lorenz_RK4}(a) shows that the estimators of $c_1$ for each entry of $(x_1,x_2,x_3)$ converge at an almost perfect rate $(NM)^{-1/2}$. Figure~\ref{fig:conv_res_Lorenz_RK4}(b) shows that the estimator of $c_0,c_1,c_2$ remain little varied until $\deltat=0.12$ (i.e., $\Gap>240$) for each entry. It also shows that the residuals of all three entries decay at a rate slightly higher than $O(\deltat^{1/2})$. 
\begin{figure}[htp!]
 \centering \vspace{-1mm}
    \subfigure[Convergence of $c_1$ in sample size]{
 \includegraphics[width=0.37\textwidth]{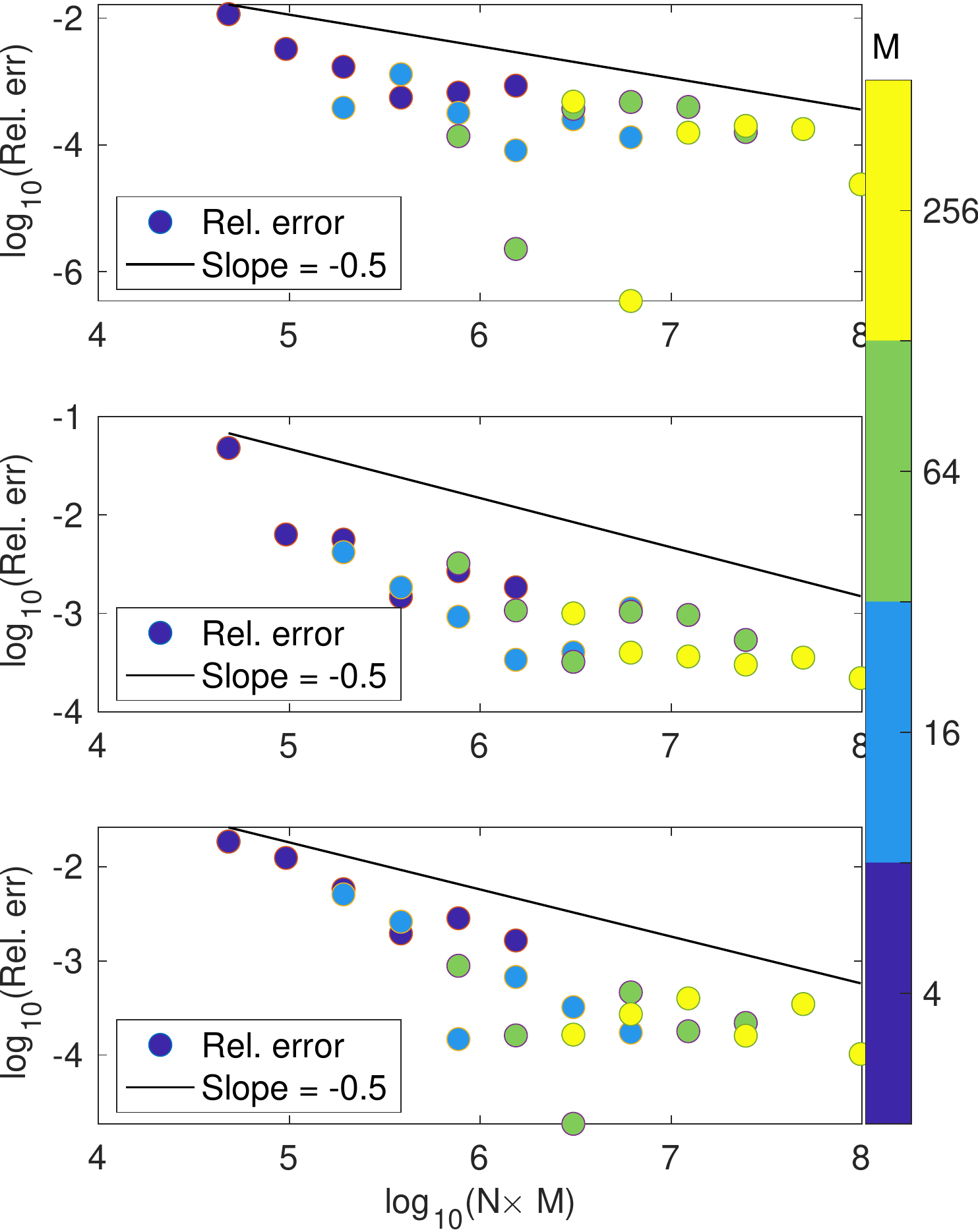}}
    \subfigure[Coefficients and residuals]{
    \includegraphics[width=0.59 \textwidth]{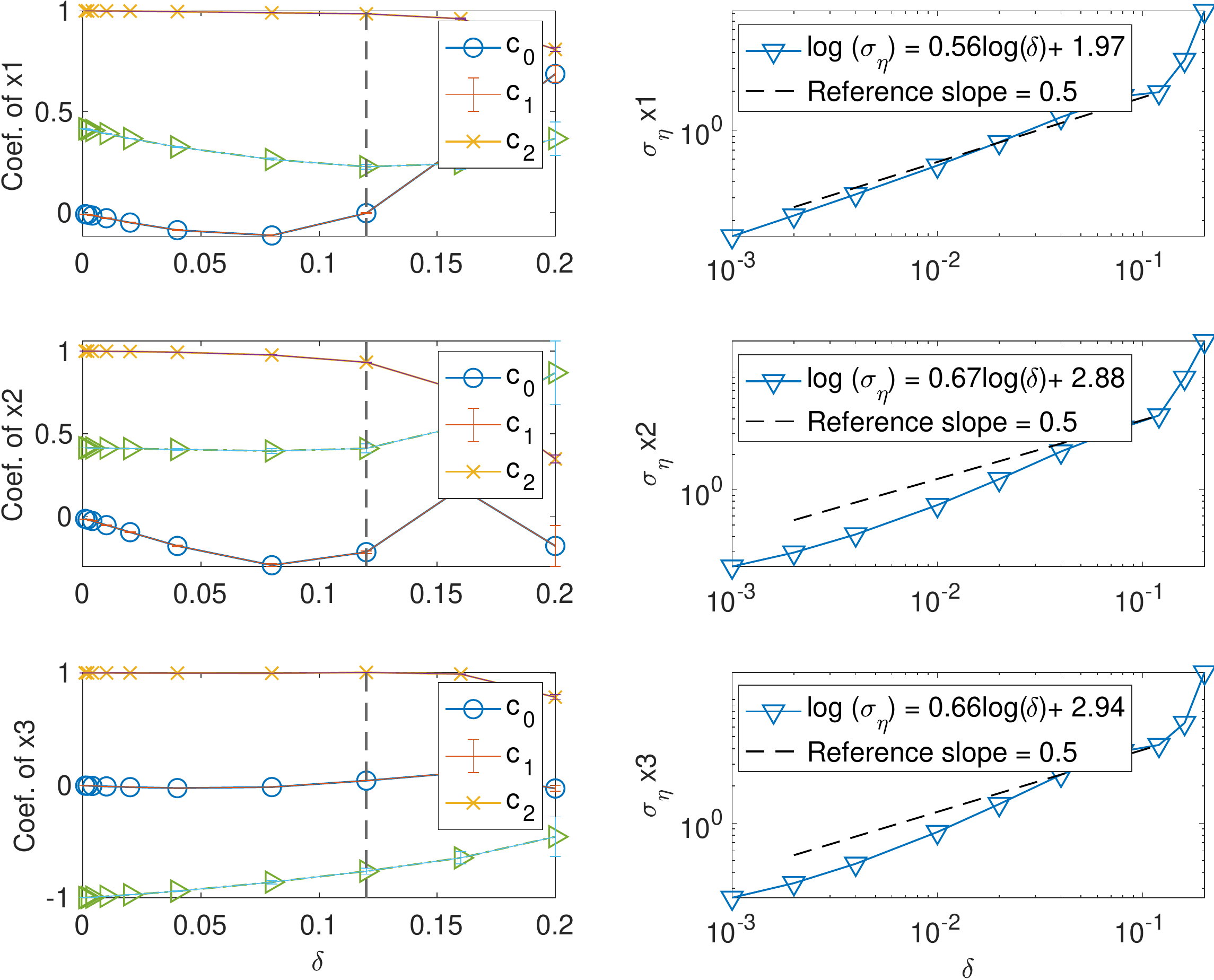}}
   
 \vspace{-3mm}     \caption{The 3D stochastic Lorenz system: Convergence of estimators in IS-RK4 with $c_0$ included. (a) The relative error of the estimator   $\widehat{c_1^{\deltat, N,M}}$ with $\deltat = 240 \Delta t = 0.12$ converges at order about $(MN)^{-1/2}$, matching Theorem \ref{thm_convEst}. (b) Left column: The estimators of $c_0,c_1, c_2$ are varies little until $\deltat>0.12$. The vertical dash line is the optimal time gap.  Right column: The residuals decay at orders slightly higher than $O(\deltat^{1/2})$. 
    \label{fig:conv_res_Lorenz_RK4} }
\end{figure}

\section{Conclusions and outlook}
We have introduced a framework to infer schemes adaptive to large time-stepping (ISALT) from data for locally Lipschitz ergodic SDEs. We formulate it as a statistical learning problem, in which we learn an approximation to the infinite-dimensional discrete-time flow map. By deriving informed basis functions from classical numerical schemes, we obtain a low-dimensional parameter estimation problem, avoiding the curse of dimensionality in statistical learning.

Under mild conditions, we show that the estimator converges as the data size increases, and the inferred scheme has same a 1-step strong order as the explicit scheme it parametrizes. Thus, our algorithm comes with performance guarantee. Numerical tests on three non-globally Lipschitz examples confirm the theory. The inferred scheme can tolerate large time-steps and efficiently and accurately simulate the invariant measure.

Many fronts are left open for further investigation. (1) The optimal time-step. We have observed that the inferred schemes perform the best (producing the most accurate invariant measure) when the time-step is medium-large. This observation suggests a trade-off between the 1-step approximation error of the flow map and the accumulated numerical error in the invariant measure. Similar optimality in the medium range was observed in space-time model reduction \cite{Lu20Burgers} and in parameter estimation for multiscale diffusion \cite{PS07}. It is crucial to have a universal \emph{a priori} estimate on the optimal time-step, which can guide all data-driven model reduction approaches.  (2) Multi-step noise. We focused on approximate flow maps that use only the increments of the Brownian motion.  This limits the performance of the inferred scheme because we omit the details of the stochastic force. Thus, a multi-step noise provides the necessary information for further improvements, particularly when the noise is non-stationary \cite{Li_Duan21}. (3) Non-ergodic systems and/or space-time reduction. We expect to extend the framework of ISALT to simulate non-ergodic systems or achieve space-time reduction for high-dimensional nonlinear systems by extracting informed basis functions from the classical numerical scheme. 

\bigskip
\noindent \textbf{Acknowledgments }
 FL is grateful for supports from NSF-DMS 1913243 and NSF-DMS 1821211. XL is grateful for supports from NSF DMS CAREER-1847770. FY is grateful for supports from AMS-Simons travel grants.   FL would like to thank Kevin Lin for helpful discussions. 

\bibliographystyle{alpha}
\let\oldbibliography\thebibliography
\renewcommand{\thebibliography}[1]{%
  \oldbibliography{#1}%
  \setlength{\itemsep}{0pt}%
}
\bibliography{ref_FeiLU,ref_LearnDiscretization}

\end{document}